\documentclass[12pt,oneside,a4paper]{amsart}
\usepackage{enumerate}
\usepackage{xcolor}
\usepackage{amsmath,amsthm,amssymb}
\usepackage{hyperref}
\usepackage[hmargin=3cm,tmargin=4.5cm,bmargin=3cm]{geometry}
\usepackage{pstricks,pst-node,pst-plot,multido}

\let\top\relax

\theoremstyle{theorem}

\newtheorem*{MT}{Main Theorem}

\newtheorem{Theorem}{Theorem}
\newtheorem{Lemma}[Theorem]{Lemma}
\newtheorem{Proposition}[Theorem]{Proposition}

\theoremstyle{definition}

\newtheorem{Definition}[Theorem]{Definition}

\DeclareMathOperator{\Pic}{Pic} \DeclareMathOperator{\Spec}{Spec}
\DeclareMathOperator{\Sing}{Sing} \DeclareMathOperator{\Gal}{Gal}
\DeclareMathOperator{\rank}{rank} \DeclareMathOperator{\diag}{diag}
\DeclareMathOperator{\Aut}{Aut}
\DeclareMathOperator{\id}{id}
\DeclareMathOperator{\tr}{tr}
\DeclareMathOperator{\top}{top}
 \DeclareMathOperator{\ord}{ord}

\DeclareMathOperator{\NS}{NS}

\everypsbox{\scriptsize}

\newcommand{\labelA}{%
 \multido{\na=-6+5,\nb=-4+5,\N=1+1}{3}{%
    \rput(\na,-7.6){$F_\N$}
    \rput(7.6,\nb){$G_\N$}
  }
  \multido{\na=-3.2+5.0,\nA=1+1}{3}{%
  \multido{\nb=-5+5.0,\nB=1+1}{3}{%
    \rput(\na,\nb){$E_{\nA\nB}$}
  }
  }
  \multido{\na=-5+5.0,\nA=1+1}{3}{%
  \multido{\nb=-6.6+5.0,\nB=1+1}{3}{%
    \rput(\na,\nb){$E_{\nA\nB}'$}
  }
  }
}

\newcommand{\labelB}{%
 \multido{\na=-6+5,\nb=-4+5,\nc=-5+5,\N=1+1}{3}{%
   \rput(\na,-7.6){$F_\N$}
   \rput(7.6,\nb){$G_\N$}
   \rput(0,\nc){$E_{2\N}'$}
 }
 \multido{\na=-4+10,\nc=-5.2+10.0,\nx=-3.5+10.0,\NA=1+2}{2}{%
 \multido{\nb=-2.6+10.0,\nd=-6+10,\ny=-6.5+10.0,\NB=1+2}{2}{%
   \rput(\na,\nb){$E_{\NA\NB}$}
   \rput(\nc,\nd){$E_{\NA\NB}'$}
   \rput(\nx,\ny){$H_{\NA\NB}$}
 }
 }
 \multido{\n=-5+10,\N=1+2}{2}{%
   \rput(\n,0){$E_{\N2}$}
 }
 \rput(-2,2){$E_{22}$}
}

\begin{document}

\title{The Classification of log Enriques Surfaces of rank 18}

\author{Fei Wang}

\address{
  Fei Wang\\
  Department of Mathematics\\
  National University of Singapore\\
  10 Lower Kent Ridge Road\\
  Singapore 119076
}

\email{matwf@nus.edu.sg}



\maketitle

\begin{abstract}
  Log Enriques surface is a generalization of K3 and Enriques
  surface. We will classify all the rational log Enriques
  surfaces of rank 18 by giving concrete models for the realizable
  types of these surfaces.
\end{abstract}

\section{Introduction}\label{sec2.1}

A normal projective surface $Z$ with at worst quotient singularities is called a \emph{logarithmic} (abbr.\ \emph{log}) \emph{Enriques surface} if its canonical Weil divisor $K_Z$ is numerically equivalent to zero, and if its irregularity $\dim H^1(Z, \mathcal O_Z)=0$. By the abundance for surfaces, $K_Z\sim_{\mathbb Q} 0$.

Let $Z$ be a log Enriques surface and define
\begin{equation*}
  I:=I(Z)= \min\{n\in \mathbb Z^+\mid \mathcal O_Z(nK_Z)\simeq \mathcal O_Z\}
\end{equation*}
to be the \emph{canonical index} of $Z$. The \emph{canonical cover} of $Z$ is defined as
\begin{equation*}
  \pi: \bar S:=\Spec_{\mathcal O_Z} \left( \bigoplus_{j=0}^{I-1} \mathcal O_Z(-jK_Z)\right)\to Z.
\end{equation*}
This is a Galois $\mathbb Z/I\mathbb Z$-cover. So $\bar S/(\mathbb Z/I\mathbb Z)=Z$.

Note that a log Enriques surface is irrational if and only if it is a K3 or Enriques surface with at worst Du Val singularities (cf.\ \cite[Proposition~1.3]{zhang1991logarithmic}). More precisely, a log Enriques surface of index one is a K3 surface with at worst Du Val singularities, and a log Enriques surface of index two is an Enriques surface with at worst Du Val singularities or a rational surface. Therefore, the log Enriques surfaces can be viewed as
generalizations of K3 surfaces and Enriques surfaces. More results about the canonical indices are studied in \cite{zhang1991logarithmic} and \cite{zhang1993logarithmic}.

\medskip

If a log Enriques surface $Z$ has Du Val singularities, let $\widetilde Z\to Z$ be  the
partial minimal resolution of all Du Val singularities of $Z$, then
$\widetilde Z$ is again a log Enriques surface of the same canonical
index as $Z$. Therefore, we assume throughout this paper that $Z$ has
no Du Val singularities; otherwise we consider $\widetilde Z$
instead.

By the definition of the canonical cover and the classification result of surfaces, we have the following (cf.\ \cite{zhang1991logarithmic}).

1. $\bar S$ has at worst Du Val singularities, and its canonical divisor $K_{\bar S}$ is linearly equivalent to zero. So $\bar S$ is either an abelian surface or a projective K3 surface with at worst Du Val singularities.

2. $\pi: \bar S\to Z$ is a finite, cyclic Galois cover of degree $I=I(Z)$, and it is \'{e}tale over $Z\backslash\Sing Z$.

3. $\Gal(\bar S/Z)\simeq \mathbb Z/I\mathbb Z$ acts faithfully on $H^0(\mathcal O_{\bar S}(K_{\bar S}))$. In other words, there is a generator $g$ of $\Gal(\bar S/Z)$ such that $g^*\omega_{\bar S} = \zeta_I \omega_{\bar S}$, where $\zeta_I$ is the $I$th primitive root of unity and $\omega_{\bar S}$ is a nowhere vanishing regular $2$-form on $\bar S$.

\medskip

Suppose $\Sing \bar S\neq \emptyset$. Let $\nu: S\to \bar S$ be the
minimal resolution of $\bar S$, and $\Delta_S$  the exceptional
divisor of $\nu$. Then $\Delta_S$ is a disconnected sum of divisors
of Dynkin's type:
\begin{equation*}
  \left(\oplus A_\alpha\right) \oplus \left(\oplus D_\beta\right) \oplus \left(\oplus E_\gamma\right)
\end{equation*}

Note that $S$ is a K3 surface. The Chern map $c_1: \Pic(S)\to
H^2(S,\mathbb Z)$ is injective. So $\Pic(S)$ is mapped
isomorphically onto the Neron-Severi group $\NS(S)$. We can therefore
define the \emph{rank} of $\Delta_S$ to be the rank of the sublattice of
the N\'{e}ron Severi lattice $\NS(S)\simeq\Pic(S)$ generated by the
irreducible components of $\Delta_S$. In other words,
\begin{equation*}
  \rank \Delta_S= \sum \alpha+\sum \beta+ \sum \gamma.
\end{equation*}
Moreover, let $\rho(S):=\rank \Pic(S)$ be the Picard number of $S$,
then
\begin{equation*}
  \rank \Delta_S \leq \rho(S)-1 \leq 20-1=19.
\end{equation*}
Since $S$ is uniquely determined up to isomorphism, by abuse of language we also say $Z$ is of type $(\oplus A_{\alpha})\oplus (\oplus D_\beta)\oplus (\oplus E_\gamma)$, and call $\rank \Delta_S$ the \emph{rank} of $Z$.

\medskip

A rational log Enriques surface $Z$ is called \emph{extremal} if
it is of rank 19, the maximal possible value $19$. The extremal log Enriques surfaces are
completely classified in \cite{oguiso1999complete}. In \cite{oguiso1998extremal}, the isomorphism classes of rational log Enriques surfaces of type $A_{18}$ and $D_{18}$ are determined. In this paper, we are going to classify all the rational log Enriques surfaces of rank 18 by proving the following theorem.

\begin{MT}\label{theorem2}
  Let $Z$ be a rational log Enriques surfaces of rank $18$ without Du Val singularities. Let $\bar S\to Z$ be the canonical cover, and $S\to \bar S$ the minimal resolution with exceptional divisor $\Delta_S$. Then we have the following assertions.
  \begin{enumerate}[\quad \rm 1)]
    \item The canonical index $I(Z)=2$, $3$ or $4$.
    \item If $I(Z)=2$, then $(S,g)\simeq (S_2,g_2)$, and $\Delta_S$ is of one of the following $5$ types:
    \begin{equation*}
      A_1\oplus A_{17},\quad A_3\oplus A_{15},\quad A_5\oplus A_{13},\quad A_7\oplus A_{11},\quad A_9\oplus A_9.
    \end{equation*}
    Moreover, all of them are realizable.
    \item If $I(Z)=3$, then $(S,g)\simeq (S_3,g_3)$,  and $\Delta_S$ is of one of the $48$ possible types in Table~\ref{table2.1}, and  from which $40$ types have been realized.
    \item If $I(Z)=4$, then $(S,g^2)\simeq (S_2,g_2)$, and $\Delta_S$ is of one of the following 3 types:
    \begin{equation*}
      A_1\oplus A_{17},\quad A_5\oplus A_{13},\quad A_9\oplus A_9.
    \end{equation*}
    Moreover, all of them are realizable.
    \item For each of the possible cases in (2) and (3), every irreducible curve in $\Delta_S$ is $g$-stable, and the action of $g$ on $\Delta$ is uniquely determined, which are given in Table~\ref{table2.2} and
    \ref{table2.1}, respectively.
  \end{enumerate}
  Here $(S_2,g_2)$ \emph{(}Definition \ref{d2.6}\emph{)} and $(S_3,g_3)$ \emph{(}Definition \ref{d2.3}\emph{)} are the Shioda-Inose's pairs of discriminants $4$ and $3$ respectively.
\end{MT}

\thanks{\textsc{Acknowledgement}. The author would like to thank for Prof D.-Q.\ Zhang for his kind guidance of the paper, and thank the referee for the valuable commands.}

\section{Preliminaries}\label{sec2.2}

\begin{Definition}\label{d2.1}
  Let $Z$ be a normal projective surface defined over the complex number field $\mathbb C$. It is called a \emph{log Enriques surface} of \emph{canonical index} $I$ if
  \begin{enumerate}[\quad 1)]\itemsep=1mm
    \item $Z$ has at worst quotient singularities, and
    \item $IK_Z$ is linearly equivalent to zero for the minimum positive integer $I$, and
    \item the irregularity $q(Z):=\dim H^1(Z,\mathcal O_Z)=0$.
  \end{enumerate}
\end{Definition}

We will use the following notations in Section~\ref{sec2.3}--\ref{sec2.4}.

\begin{enumerate}[\quad 1.]\itemsep=1mm
  \item For each $I\in \mathbb Z^+$, $\zeta_I=\exp(2\pi\sqrt{-1}/I)$, a primitive $I$th root of unity.
  \item Let $X$ be a variety, and $G$ an automorphism group on $X$. For each $g\in X$, denote the fixed locus by $X^g=\{x\in X\mid g(x)=x\}$. Set $X^{[G]}= \bigcup_{g\in G\backslash \{\id\}} X^g$.
  \item Let $S$ be a surface and $g$ an automorphism on $S$. A curve $C$ on $S$ is called \emph{$g$-stable} if $g(C)=C$, and it is called \emph{$g$-fixed} if $g(x)=x$ for every $x\in C$. A point $x\in S$ is an \emph{isolated $g$-fixed point} if $g(x)=x$ and it is not contained in any $g$-fixed curve.
\end{enumerate}

\section[Shioda-Inose's Pairs]{Log Enriques Surfaces from Shioda-Inose's Pairs}\label{sec2.3}

In this section, we assume that $Z$ is a rational log Enriques surface of rank $18$ and
canonical index $I$ without Du Val singularities. Let $\pi: \bar S\to Z$ be the  canonical cover
of $Z$, and $\nu: S\to \bar S$ the minimal resolution of $\bar S$ with exceptional divisor $\Delta_S$. Then
\begin{equation*}
  20\geq \rho(S)\geq \rank \Delta_S+1=19.
\end{equation*}
Recall that $S$ is a K3 surface. Let $T_S$ denote the transcendental lattice of $S$, i.e., the orthogonal complement of $\Pic(S)$ in
$H^2(S,\mathbb Z)$. Then
\begin{equation*}
  \rank T_S= \dim H^2(S,\mathbb
Z)-\rho(S)=22- \rho(S)=2\ \textrm{or}\ 3.
\end{equation*}

\medskip
Let $g$ be the automorphism on $S$ induced by a generator of
$\Gal(\bar S/Z)$, and $\omega_S$ a nowhere vanishing holomorphic $2$-form on
$S$. Then $g^*\omega_S= \zeta_I\omega_S$. Note that $\omega_S \in
T_S \otimes \mathbb C$. So $\zeta_I$ is an eigenvalue of $g^*$
acting on $T_S$. Therefore, $\varphi(I)\leq \rank T_S\leq 3$, where
$\varphi$ is Euler's phi function. It follows that

\begin{Lemma}\label{l2.2}
  The canonical index $I(Z)=2,3,4$ or $6$.
\end{Lemma}

We have indicated that all the realizable rational log Enriques surfaces listed in Main Theorem can be constructed from the Shioda-Inose's pairs $(S_2,g_2)$ or $(S_3,g_3)$ (cf.~\cite{shioda-singular}). Precisely, if $I(Z)=2$, then $(S,g)\simeq (S_2,g_2)$; if $I(Z)=3$,
then $(S,g)\simeq (S_3,g_3)$; if $I(Z)=4$, then $(S,g^2)\simeq
(S_2,g_2)$; we will also show that $I\neq 6$.

\begin{Definition}\label{d2.3}
  Let $\zeta_3:=\exp(2\pi\sqrt{-1}/3)$, and $E_{\zeta_3}:= \mathbb C/(\mathbb Z+\mathbb Z\zeta_3)$ the elliptic curve of period $\zeta_3$. Let $\bar S_3:= E_{\zeta_3}^2 / \langle \diag(\zeta_3,\zeta_3^2) \rangle$ be the quotient surface, and $S_3\to \bar S_3$ the minimal resolution of $\bar S_3$. Let $g_3$ be the automorphism of $S_3$ induced by the action $\diag(\zeta_3,1)$ on $E_{\zeta_3}^2$. Then $(S_3,g_3)$ is called the \emph{Shioda-Inose's pair of discriminant $3$}.
\end{Definition}

\begin{figure}[h!]
\begin{center}\psset{unit=0.5}
\begin{pspicture*}(-7,-8)(8,7)

\labelA

  \psline[linewidth=2pt](-7,-4)(7,-4) %G1
  \psline[linewidth=2pt](-7,1)(7,1) %G2
  \psline[linewidth=2pt](-7,6)(7,6) %G3

  \psline[border=2mm,linewidth=2pt](-6,-7)(-6,7) %F1
  \psline[border=2mm,linewidth=2pt](-1,-7)(-1,7) %F2
  \psline[border=2mm,linewidth=2pt](4,-7)(4,7) %F3

  \psline(-7,4)(-3,4) %E13'
  \psline(-2,4)(2,4) %E23'
  \psline(3,4)(7,4) %E33'

  \psline(-7,-1)(-3,-1) %E12'
  \psline(-2,-1)(2,-1) %E22'
  \psline(3,-1)(7,-1) %E32'

  \psline(-7,-6)(-3,-6) %E11'
  \psline(-2,-6)(2,-6) %E21'
  \psline(3,-6)(7,-6) %E31'

  \psline(-4,-7)(-4,-3) %E11
  \psline(-4,-2)(-4,2) %E12
  \psline(-4,3)(-4,7) %E13

  \psline(1,-7)(1,-3) %E21
  \psline(1,-2)(1,2) %E22
  \psline(1,3)(1,7) %E23

  \psline(6,-7)(6,-3) %E31
  \psline(6,-2)(6,2) %E32
  \psline(6,3)(6,7) %E33
\end{pspicture*}
\end{center}
\caption{$(S_3,g_3)$}\label{fig2.1}
\end{figure}

It is proved in \cite{oguiso1996most} and \cite{oguiso1999complete} that

\begin{Proposition}\label{p2.4}
  Let $(S_3,g_3)$ be the Shioda-Inose's pair of discriminant $3$. Then
  \begin{enumerate}[\quad \rm 1)]\itemsep=1mm
    \item
    $S_3$ contains $24$ rational curves: $F_1,F_2,F_3$ coming from $(E_{\zeta_3})^{\zeta_3}\times E_{\zeta_3}$; $G_1,G_2,G_3$ coming from $E_{\zeta_3}\times (E_{\zeta_3})^{\zeta_3}$; and $E_{ij}, E_{ij}'$ $(i,j=1,2,3)$ the exceptional curves arising from the $9$ Du Val singular points of $\bar S_3$ {\rm(}Figure.~\ref{fig2.1}\emph{)};
    \item
    $g_3^*\omega_{S_3}= \zeta_3\omega_3$, where $\omega_{S_3}$ is a nowhere vanishing holomorphic $2$-form on $S_3$, and $g_3^*|_{\Pic(S_3)}=\id$; so each of the $24$ curves is $g_3$-stable;
    \item
    $S_3^{g_3}= (\coprod_{i=1}^3 F_i) \coprod (\coprod_{j=1}^3 G_j) \coprod (\coprod_{i,j=1}^3 \{P_{ij}\})$, where $\{P_{ij}\}=E_{ij}\cap E_{ij}'$;
    \item
    $g_3\circ \varphi= \varphi\circ g_3$ for all $\varphi\in \Aut(S_3)$.
  \end{enumerate}
\end{Proposition}

\begin{Proposition}\label{p2.5}
  Let $(S,g)$ be a pair of a smooth K3 surface $S$ and an automorphism of $g$ on $S$. Assume that
  \begin{enumerate}[\quad \rm 1)]\itemsep=1mm
    \item $g^3=\id$, the identity on $S$;
    \item $g^*\omega_S= \zeta_3 \omega_S$, where $\omega_S$ is a nowhere vanishing holomorphic $2$-form on $S$;
    \item $S^g$ consists of only rational curves and isolated points;
    \item $S^g$ contains at least $6$ rational curves.
  \end{enumerate}
Then $(S,g)\simeq (S_3,g_3)$. Moreover, $S^g$ consists of exactly $6$ rational curves and $9$ isolated points.
\end{Proposition}

\begin{Definition}\label{d2.6}
  Let $E_{\zeta_4}:= \mathbb C/(\mathbb Z+\mathbb Z\sqrt{-1})$ be the elliptic curve of period $\zeta_4=\sqrt{-1}$. Let $\bar S_2:= E_{\zeta_4}^2/\langle\diag(\zeta_4,\zeta_4^3)\rangle$ be the quotient surface and $S_2\to \bar S_2$ the minimal resolution of $\bar S_2$. Let $g_2$ be the involution of $S_2$ induced by the action $\diag(-1,1)$ on $E_{\zeta_4}^2$. Then $(S_2,g_2)$ is called the \emph{Shioda-Inose's pair of discriminant $4$}.
\end{Definition}

\begin{figure}[h!]
\begin{center}\psset{unit=0.5}
\begin{pspicture*}(-7,-8)(8,8)

  \psline[linewidth=2pt](-7,-4)(7,-4) %G1
  \psline[linewidth=2pt](-7,1)(7,1) %G2
  \psline[linewidth=2pt](-7,6)(7,6) %G3

  \psline[border=2mm,linewidth=2pt](-6,-7)(-6,7) %F1
  \psline[border=2mm,linewidth=2pt](-1,-7)(-1,7) %F2
  \psline[border=2mm,linewidth=2pt](4,-7)(4,7) %F3

  \psline(-4,7)(-3,4)  %E13
  \psline(-7,4)(-4,3)  %E13'
  \psline[linewidth=2pt](-5,3)(-3,5)  %H13

  \psline(6,7)(7,4)  %E33
  \psline(3,4)(6,3)  %E33'
  \psline[linewidth=2pt](5,3)(7,5)  %H33

  \psline(-4,-3)(-3,-6)  %E11
  \psline(-7,-6)(-4,-7)  %E11'
  \psline[linewidth=2pt](-5,-7)(-3,-5)  %H11

  \psline(6,-3)(7,-6)  %E31
  \psline(3,-6)(6,-7)  %E31'
  \psline[linewidth=2pt](5,-7)(7,-5)  %H31

  \psline(-2,4)(1,7) %E23'
  \psline(-2,-1)(1,2) %E22'
  \psline(-3,0)(0,3)  %E22
  \psline(-7,-1)(-4,2)  %E12
  \psline(3,-1)(6,2)  %E32
  \psline(-2,-6)(1,-3) %E21'

  \labelB
\end{pspicture*}
\end{center}\caption{$(S_2,g_2)$}\label{fig2.2}
\end{figure}

It is also proved in \cite{oguiso1996most} and \cite{oguiso1999complete} that

\begin{Proposition}\label{p2.7}
  Let $(S_2,g_2)$ be the Shioda-Inose's pair of discriminant $4$. Then
  \begin{enumerate}[\quad \rm 1)]\itemsep=1mm
  \item
  $S_2$ contains $24$ rational curves: $F_1,F_2,F_3$ coming from $(E_{\zeta_4})^{[\langle \zeta_4\rangle]} \times E_{\zeta_4}$; $G_1,G_2,G_3$ coming from $E_{\zeta_4}\times (E_{\zeta_4})^{[\langle \zeta_4\rangle]}$; and $E_{ij}'+ H_{ij}+E_{ij}$, $i,j\in \{1,3\}$, the exceptional curves arising from the $4$ Du Val singular points of Dynkin type $A_3$; and $E_{12}, E_{22}, E_{32}, E_{21}', E_{22}', E_{23}'$, the exceptional curves arising from the $6$ Du Val singular points of Dynkin type $A_1$ {\rm(}Figure.~\ref{fig2.2}{\rm)};
  \item
  $g_2^*\omega_{S_2} = -\omega_{S_2}$, where $\omega_{S_2}$ is a nowhere vanishing holomorphic $2$-form on $S_2$, and $g_2^*|_{\Pic(S)}=\id$; so each of the $24$ curves is $g_2$-stable;
  \item
  $S_2^{g_2}= (\coprod_{i=1}^3 F_i) \coprod (\coprod_{j=1}^3 G_j) \coprod (\coprod_{i,j\in \{1,3\}} H_{ij})$;
  \item
  $g_2\circ \varphi= \varphi\circ g_2$ for all $\varphi\in \Aut(S_2)$.
\end{enumerate}
\end{Proposition}

\begin{Proposition}\label{p2.8}
  Let $(S,g)$ be a pair of a smooth K3 surface $S$ and an automorphism $g$ of $S$. Assume that
  \begin{enumerate}[\quad \rm 1)]\itemsep=1mm
    \item $g^2=\id$, the identity on $S$;
    \item$g^*\omega_S= -\omega_S$, where $\omega_S$ is a nowhere vanishing holomorphic $2$-form on $S$;
    \item $S^g$ consists of only rational curves;
    \item $S^g$ contains at least $10$ rational curves.
  \end{enumerate}
Then $(S,g)\simeq (S_2,g_2)$. Moreover, $S^g$ consists of exactly $10$ rational curves.
\end{Proposition}

\section{The Classification}\label{sec2.4}

In this section, we assume that $Z$ is a log Enriques surface of rank 18 without Du Val singularities. Let $\pi: \bar S\to Z$ be the canonical cover, and $\nu: S\to \bar S$ the minimal resolution with exceptional divisor $\Delta:=\Delta_S$. Since the canonical cover $\bar S\to Z$ is unramified in codimension
one,  every curve in $S^{[\langle g\rangle]}$ is contained in
$\Delta$. In particular, $S^{[\langle g\rangle]}$ consists
of only smooth rational curves and a finite number of isolated
points, and $\Delta$ is $g$-stable.

\medskip

In general, let $S$ be a K3 surface, and $g$ an automorphism of $S$ of order $n$. Let $T_S$ be its transcendental lattice. Note that $g$ induces
actions $g^*$ on $\Pic(S)\otimes \mathbb C$ and on $T_S\otimes
\mathbb C$. Since $g^n=\id$, these actions are diagonalizable and
every eigenvalue of $g^*$ is an $n$th root of unity, say $\zeta_n^i$
for some $0\leq i<n$. Since $g^*$ is well-defined on $\Pic(S)$ and $T_S$, the number of
eigenvalues $\zeta_n^i$ of $g^*|_{\Pic(S)\otimes \mathbb C}$ and
$g^*|_{T_S\otimes \mathbb C}$ equals to that of the conjugate
eigenvalues $\bar \zeta_n^i$, respectively. By noting that $\dim
H^2(S,\mathbb C)=22$, we have the following lemma:

\begin{Lemma}[{\cite[Lemma 2.0]{oguiso1996most}}]\label{l2.9}
With the notations above, let $t_0$ and $r_0$ be the rank of the invariant lattices
$(\Pic(S))^{g^*}$ and $(T_S)^{g^*}$, respectively. Let $I_s$ denote the identity matrix of size $s$.
\begin{enumerate}[\quad \rm 1)]\itemsep=1mm
\item If $n=2k+1$ is odd, then $\rho(S)=t_0+ 2\sum_{i=1}^k t_i$ and
\begin{align*}
  g^*|_{\Pic(S)\otimes \mathbb C} &= \diag(I_{t_0}, \zeta_n I_{t_1}, \bar \zeta_n I_{t_1}, \zeta_n^2 I_{t_2}, \bar \zeta_n^2 I_{t_2}, \ldots, \zeta_n^k I_{t_k}, \bar \zeta_n^k I_{t_k}),\\
  g^*|_{T_S\otimes \mathbb C} & = \diag(I_{r_0}, \zeta_n I_{r_1}, \bar \zeta_n I_{r_1}, \zeta_n^2 I_{r_2}, \bar \zeta_n^2 I_{r_2}, \ldots, \zeta_n^k I_{r_k}, \bar \zeta_n^k I_{r_k}),
\end{align*}
and $t_0+r_0+2\sum_{i=1}^k t_i + 2\sum_{i=1}^k r_i=22$.

\item If $n=2k$ is even, then $\rho(S)=t_0+2\sum_{i=1}^{k-1}t_i +t_k$ and
\begin{align*}
  g^*|_{\Pic(S)\otimes \mathbb C} & = \diag(I_{t_0}, \zeta_n I_{t_1}, \bar \zeta_n I_{t_1}, \zeta_n^2 I_{t_2}, \bar \zeta_n^2 I_{t_2}, \ldots, \zeta_n^{k-1} I_{t_{k-1}}, \bar \zeta_n^{k-1} I_{t_{k-1}}, -I_{t_k}),\\
  g^*|_{T_S\otimes \mathbb C} & = \diag(I_{r_0}, \zeta_n I_{r_1}, \bar \zeta_n I_{r_1}, \zeta_n^2 I_{r_2}, \bar \zeta_n^2 I_{r_2}, \ldots, \zeta_n^{k-1} I_{r_{k-1}}, \bar \zeta_n^{k-1} I_{r_{k-1}}, -I_{r_k}),
\end{align*}
and $t_0+r_0+ 2\sum_{i=1}^{k-1} t_i + 2\sum_{i=1}^k r_i + t_k+r_k=22$.
\end{enumerate}
\end{Lemma}

\section*{4.1. Classification When $I=3$}\label{sec2.4.1}

Let $(S,g)$ be a pair of smooth K3 surface $S$ and an automorphism
$g$ of $S$. We assume that $g^*\omega_S= \zeta_3 \omega_S$ for a nowhere
vanishing holomorphic $2$-form $\omega_S$ on $S$.

%Suppose $g^*x=x$ for some $x\in T_S$, the transcendental lattice. We have
%\begin{equation*}
%  x\cdot \omega_S= g^*(x\cdot \omega_S) =g^*x\cdot g^*\omega_S= \zeta_3x\cdot \omega_S.
%\end{equation*}
%Then $x\cdot \omega_S=0$, and so $x\in \Pic(S)\cap T_S=\{0\}$.
%Hence, $(T_S)^{g^*}=\{0\}$ and $r_0=0$ as defined in
%Lemma~\ref{l2.9}. Therefore,  $r_1=1$ and $g^*|_{T_S\otimes \mathbb
%C}=\diag(\zeta_3,\bar \zeta_3)$. In particular, $\rho(S)=20$, i.e.,
%$S$ is a singular K3 surface.

Let $P$ be an isolated $g$-fixed point on $S$. Then $g^*$ can be
written  as $\diag(\zeta_3^a,\zeta_3^b)$ for some $a,b\in \{1,2\}$
with $a+b\equiv1 \pmod 3$ under some appropriate local coordinates
around $P$ because $g^*\omega_S= \zeta_3\omega_S$. We see that
$a=b=2$ and the action is $\diag(\zeta_3^2,\zeta_3^2)$. If $C$ is a $g$-fixed irreducible curve and $Q\in C$, then it also
follows from $g^*\omega_S=\zeta_3\omega_S$ that $g^*$ can be written
as $\diag(1,\zeta_3)$ under some appropriate local coordinates
around $Q$. In particular, the $g$-fixed curves are smooth and
mutually disjoint.

We need to use the following lemma in the classification for $I=3$.

\begin{Lemma}[{``Three Go'' Lemma, \cite[Lemma~ 2.2]{oguiso1996most}}]\label{l2.10} Let $(S,g)$ be a pair of smooth K3 surface $S$ and an automorphism $g$ of $S$. Assume that $g^3=\id$ and $g^*\omega_S=\zeta_3\omega_S$.
  \begin{enumerate}[\quad \rm 1)]\itemsep=1mm
    \item Let $C_1-C_2-C_3$ be a linear chain of $g$-stable smooth rational curves. Then exactly one of $C_i$ is $g$-fixed.
    \item Let $C$ be a $g$-stable but not $g$-fixed smooth rational curve. Then there is a unique $g$-fixed curve $D$ such that $C\cdot D=1$.
    \item Let $M$ and $N$ be the number of smooth rational curves and the number of isolated points in $S^g$, respectively. Then $M-N=3$.
  \end{enumerate}
\end{Lemma}

\medskip

Suppose $I(Z)=3$. Then the associated pair $(S,g)$ satisfies the conditions in Lemma~\ref{l2.10}. We first determine a possible list of the Dynkin's types of $\Delta$.

\begin{Proposition}\label{p2.11} With the notations as in Main Theorem, suppose $I(Z)=3$. Then $(S,g)\simeq (S_3,g_3)$, the Shioda-Inose's pair of discriminant $3$. Moreover, $\Delta$ is of one of the following $13$ types:
  \begin{enumerate}[\quad \rm I.]\itemsep=1mm
  \item $A_{18}$;
  \item $D_{18}$;
  \item $A_{3m}\oplus A_{3n}$,\quad $m+n=6$;
  \item $D_{3m}\oplus A_{3n}$,\quad $m+n=6$;
  \item $D_{3m}\oplus D_{3n}$,\quad $m+n=6$;
  \item $D_{3m+1}\oplus A_{3n-1}$,\quad $m+n=6$;
  \item $A_{3m}\oplus A_{3n}\oplus A_{3r}$,\quad $m+n+r=6$;
  \item $D_6\oplus D_6\oplus D_6$,
  \item $A_{3m}\oplus D_{3n}\oplus D_{3r}$,\quad $m+n+r=6$;
  \item $A_{3m}\oplus A_{3n}\oplus D_{3r}$,\quad $m+n+r=6$;
  \item $D_{3m+1}\oplus A_{3n}\oplus A_{3r-1}$,\quad $m+n+r=6$;
  \item $D_{3m+1}\oplus D_{3n+1}\oplus A_{3r-2}$,\quad $m+n+r=6$;
  \item $D_{3m+1}\oplus D_{3n}\oplus A_{3r-1}$,\quad $m+n+r=6$.
  \end{enumerate}
\end{Proposition}

\begin{proof} Let $\Delta_i$ be a connected component of $\Delta$.

\medskip

Step 1: $\Delta_i$ is $g$-stable.

If $\Delta_i$ is not $g$-stable, then its image in $Z$ would be a Du
Val singular point since $I(Z)=3$ is a prime. However,  we have
assumed that $Z$ has no Du Val singularities.

\medskip

Step 2: $\Delta_i=A_n$ or $D_n$.

Suppose there is a $\Delta_i=E_n$ for some $n$. Let $C$ be the
center of $\Delta_i$, and $C_1,C_2,C_3$ the rational curves in
$\Delta_i$ which intersect $C$. Suppose $C_1$ is the twig of length one. By the uniqueness of $C$ and $C_1$, they are $g$-stable. If $C$ is
not $g$-fixed, then $\Delta_i=E_6$ and $g$ switches the other two
twigs, which contradicts $g^3=\id$. If $C$ is $g$-fixed, then each
irreducible curve in $\Delta_i$ is $g$-stable. Let $C_2-C_2'$ be a
twig of $\Delta_i$. Then $C_2'$ is not $g$-fixed and it does not
intersect with any $g$-fixed curve, which contradicts
Lemma~\ref{l2.10}.

\medskip

Step 3. Every irreducible curve in $\Delta_i$ is $g$-stable.

i) Let $\Delta_i=A_n$. Write the irreducible curves in $\Delta_i$ as
a  chain $C_1-C_2-\cdots -C_n$. For $n>1$, if $C_1$ is not
$g$-stable, we must have $g(C_1)=C_n$ and $g(C_n)=g(C_1)$, and this
contradicts $g^3=\id$.

ii) Let $\Delta_i=D_n$. Then by the uniqueness its center $C$ is
$g$-stable. Let $C_1$ and $C_2$ be twigs of length one, and $C_3$
the curve of another twig which intersects $C$.

Suppose $n>4$. Then every irreducible component in the longest twig
shall be $g$-stable. If $C_1$ is not $g$-stable, then $g(C_1)=C_2$
and $g(C_2)=C_1$, which contradicts $g^3=\id$. Thus, every
irreducible curve in $\Delta_i$ is $g$-stable. Suppose $n=4$. If $C_1$ is not $g$-stable, we must have
$g(C_1)=C_2$, $g(C_2)=C_3$ and $g(C_3)=g(C_1)$. In particular, $C$
is not $g$-fixed, and it does not intersect with any $g$-fixed curve.
This contradicts Lemma~\ref{l2.10}. Therefore, $C_1$ is $g$-stable.
We see similarly as in the case $n>4$ that $C_2$ and $C_3$ are both
$g$-stable.

\medskip

Step 4. The $g$-fixed curves of $\Delta_i$ are described as follows.

We use ``$f$'' to denote $g$-fixed curves, and ``$s$'' to denote
$g$-stable but not $g$-fixed curves in $\Delta_i$. $k$ is the number of $g$-fixed curves in $\Delta_i$.

i) Suppose $\Delta_i=A_n$.
\begin{enumerate}[\qquad a)]\itemsep=1mm
  \item $n=3k-2$:
  \begin{equation*}
  \begin{psmatrix}[rowsep=3mm,colsep=5mm,linewidth=0.5pt, nodesep=1mm]
  f & s & s & f & s & \cdots & s & s & f
  \ncline{1,1}{1,2}
  \ncline{1,2}{1,3}
  \ncline{1,3}{1,4}
  \ncline{1,4}{1,5}
  \ncline{1,5}{1,6}
  \ncline{1,6}{1,7}
  \ncline{1,7}{1,8}
  \ncline{1,8}{1,9}
  \end{psmatrix}
  \end{equation*}

  \item $n=3k-1$:
\begin{equation*}
\begin{psmatrix}[rowsep=3mm,colsep=5mm,linewidth=0.5pt, nodesep=1mm]
  f & s & s & f & s & \cdots & s & f & s
  \ncline{1,1}{1,2}
  \ncline{1,2}{1,3}
  \ncline{1,3}{1,4}
  \ncline{1,4}{1,5}
  \ncline{1,5}{1,6}
  \ncline{1,6}{1,7}
  \ncline{1,7}{1,8}
  \ncline{1,8}{1,9}
\end{psmatrix}
\end{equation*}

\item $n=3k$:
\begin{equation*}
\begin{psmatrix}[rowsep=3mm,colsep=5mm,linewidth=0.5pt, nodesep=1mm]
  s & f & s & s& f & \cdots & s & f & s
  \ncline{1,1}{1,2}
  \ncline{1,2}{1,3}
  \ncline{1,3}{1,4}
  \ncline{1,4}{1,5}
  \ncline{1,5}{1,6}
  \ncline{1,6}{1,7}
  \ncline{1,7}{1,8}
  \ncline{1,8}{1,9}
\end{psmatrix}
\end{equation*}
\end{enumerate}

ii) Suppose $\Delta_i=D_n$.

\begin{enumerate}[\qquad a)]\itemsep=1mm

\item $n=3k$:
\begin{equation*}
\begin{psmatrix}[rowsep=3mm,colsep=5mm,linewidth=0.5pt, nodesep=1mm]
  & s \\
  s & f & s & s & f & \cdots & s & s & f
  \ncline{1,2}{2,2}
  \ncline{2,1}{2,2}
  \ncline{2,2}{2,3}
  \ncline{2,3}{2,4}
  \ncline{2,4}{2,5}
  \ncline{2,5}{2,6}
  \ncline{2,6}{2,7}
  \ncline{2,7}{2,8}
  \ncline{2,8}{2,9}
\end{psmatrix}
\end{equation*}

\item $n=3k+1$:
\begin{equation*}
\begin{psmatrix}[rowsep=3mm,colsep=5mm,linewidth=0.5pt, nodesep=1mm]
  & s \\
  s & f & s & s & f & \cdots & s & f & s
  \ncline{1,2}{2,2}
  \ncline{2,1}{2,2}
  \ncline{2,2}{2,3}
  \ncline{2,3}{2,4}
  \ncline{2,4}{2,5}
  \ncline{2,5}{2,6}
  \ncline{2,6}{2,7}
  \ncline{2,7}{2,8}
  \ncline{2,8}{2,9}
\end{psmatrix}
\end{equation*}
\end{enumerate}

The case $\Delta_i=A_n$ follows from Lemma~\ref{l2.10}. Suppose
$\Delta_i=D_n$. Then by Step 3, the center $C$ is $g$-fixed. So in
the longest twig $C_3-C_4-\cdots-C_{n-1}$ of $\Delta_i$, by
induction, $C_{3j+2}$ are $g$-fixed and others are not. If $n=3k+2$
for some $k$, then $C_{n-2}$ and $C_{n-1}$ are not $g$-fixed, and
$C_{n-1}$ does not intersect with any $g$-fixed curve, a
contradiction to Lemma~\ref{l2.10}. Therefore, $n\not\equiv 2\pmod
3$.

\medskip

Step 5. $(S,g)\simeq (S_3,g_3)$.

Let $M$ be the number of isolated $g$-fixed points and $N$ the
number of $g$-fixed curves in $\Delta$. We can decompose
\begin{equation*}
  \Delta= \bigoplus_{i=1}^a D_{3\ell_i+1} \oplus \bigoplus_{i=1}^b D_{3m_i} \oplus \bigoplus_{i=1}^c A_{3p_i} \oplus \bigoplus_{i=1}^d A_{3q_i-1} \oplus \bigoplus_{i=1}^e A_{3r_i-2}.
\end{equation*}
Then
\begin{align*}
  N & = \sum_{i=1}^a \ell_i+ \sum_{i=1}^b m_i+ \sum_{i=1}^c p_i+ \sum_{i=1}^d q_i+ \sum_{i=1}^e r_i, \\
  M & \geq  \sum_{i=1}^a (\ell_i+2)+ \sum_{i=1}^b (m_i+1)+ \sum_{i=1}^c (p_i+1)+ \sum_{i=1}^d q_i+ \sum_{i=1}^e (r_i-1) \\
  & = N+(2a+b+c-e).
\end{align*}
Thus, by Lemma~\ref{l2.10}, $3=M-N\geq 2a+b+c-e$. Recall that
\begin{align*}
  \rank \Delta =18 & =  \sum_{i=1}^a (3\ell_i+1)+ \sum_{i=1}^b 3m_i+ \sum_{i=1}^c 3p_i+ \sum_{i=1}^d (3q_i-1)+ \sum_{i=1}^e (3r_i-2) \\
  & = 3N+a-d-2e.
\end{align*}
Or equivalently, $N= 6+ \dfrac{-a+d+2e}{3}$. If $N\leq 5$, then
$a\geq d+2e+3$, and we would have
\begin{equation*}
  3\geq 2a+b+c-e\geq 2(d+2e+3)+b+c-e =b+c+2d+3e+6\geq 6.
\end{equation*}
Therefore, $N\geq 6$; and hence by Proposition~\ref{p2.5}, $N=6$ and
$M=9$. Furthermore, we have $(S,g)\simeq (S_3,g_3)$.

\medskip

Step 6. Determine the Dynkin's type of $\Delta$.

Solving the system
\begin{equation*}
  d+2e=a\quad \textrm{and}\quad 2a+b+c-e\leq 3,
\end{equation*}
we have $13$ nonnegative integer solutions. So there are $13$ types
of $\Delta$ as listed  in Proposition~\ref{p2.11}.
\end{proof}

To be more precise, we list all the $48$ possible types of $\Delta$
in Table~\ref{table2.1} in Section~\ref{sec2.5}. Note that in Step 3 and 4, we proved that each irreducible curve in $\Delta$ $g$-stable, and the action of $g$ on $\Delta$ is uniquely determined, which is also included in Table~\ref{table2.1}. The case $I=3$ for Main Theorem (5) is proved.

\medskip

If $\Delta$ can be obtained
from the $24$ $g$-stable rational curves in $S_3$
(Figure~\ref{fig2.1}) which contains the $6$ $g$-fixed curves and satisfies the condition in the proof of Proposition~\ref{p2.11} Step~4, let $S_3\to \bar S$ be the contraction of
$\Delta$, then the automorphism $g_3$ on $S_3$ induces an
automorphism on $\bar S$. We see that $Z=\bar S/\langle g_3\rangle$
is a required log Enriques surface of type $\Delta$. By
verification, 40 cases are realizable. The detailed list is given in
Table~\ref{fig2.1}(A). Thus, we have completed the proof of Main Theorem~(3).

Unfortunately, the remaining 8 cases are not realizable by the the
24 curves on $S_3$, which are given in Table~\ref{fig2.1}~(B). We
are unable to determine their realizability.

%\begin{tabbing}
%  \=\hspace{8mm}\=Case V:\qquad \=(2)\quad \=$D_9\oplus D_9$,\\[2mm]
%  \>\>Case IX: \>(1) \>$A_3\oplus D_6\oplus D_9$,\\[2mm]
%  \>\>Case IX: \>(2) \>$A_6\oplus D_6\oplus D_6$,\\[2mm]
%  \>\>Case XII:\>(1) \>$D_4\oplus D_4\oplus A_{10}$,\\[2mm]
%  \>\>Case XII:\>(3) \>$D_4\oplus D_{10}\oplus A_4$,\\[2mm]
%  \>\>Case XII:\>(4) \>$D_4\oplus D_{13}\oplus A_1$,\\[2mm]
%  \>\>Case XIII:\>(1) \>$D_4\oplus D_6\oplus A_8$,\\[2mm]
%  \>\>Case XIII:\>(2) \>$D_4\oplus D_9\oplus A_5$.
%\end{tabbing}

\section*{4.2. Classification When $I=2$}\label{sec2.4.2}

Let $(S,g)$ be a pair of a smooth K3 surface $S$ and an
automorphism $g$ of $S$. We assume that $g^*\omega_S= -\omega_S$ for a nowhere vanishing holomorphic
$2$-form $\omega_S$ on $S$.

If $P\in S$ is an isolated $g$-fixed point, then $g^*$ can be
written as $\diag(-1,-1)$ under some appropriate local coordinates
around $P$. However, this contradicts the assumption that $g^*\omega_S=-\omega_S$. So $S$ has no isolated $g$-fixed point. Let $C$ be a $g$-fixed irreducible curve and let $Q\in C$. Then
$g^*$ can be written as $\diag(1,-1)$ under some appropriate local
coordinates around $Q$. So the $g$-fixed curves are smooth and
mutually disjoint.

We need to use the following lemma in the classification.

\begin{Lemma}[{``Two Go'' Lemma, \cite[Lemma~3.2]{oguiso1996most}}]\label{l2.12}
Let $(S,g)$ be a pair of smooth K3 surface and an automorphism $g$ of $S$. Assume that $g^2=\id$ and $g^*\omega_S=-\omega_S$.
\begin{enumerate}[\quad \rm 1)]
    \itemsep=1mm
    \item If $C_1-C_2$ is a linear chain of $g$-stable smooth rational curves, then exactly one of $C_i$ is $g$-fixed.
    \item If $C_1$ and $C_2$ are $g$-stable but not $g$-fixed smooth rational curves, then $C_1\cdot C_2$ is even.
    \item If $C$ is a $g$-stable but not $g$-fixed smooth rational curve, then $C$ has exactly 2 $g$-fixed
    points.
  \end{enumerate}
\end{Lemma}

Suppose $I(Z)=2$. Then the associated pair satisfies the conditions in Lemma~\ref{l2.12}. We can now determine the possible Dynkin's types of $(S,g)$.

\begin{Proposition}\label{p2.13}
  With the notations as in Main Theorem. Suppose $I=2$. Then $(S,g)\simeq (S_2,g_2)$, the Shioda-Inose's pair of discriminant $4$. Moreover, $\Delta$ is of the type $A_{2m-1}\oplus A_{2n-1}$, where $m+n=10$.
\end{Proposition}

\begin{proof}
   Since $I=2$ is a prime,
   each connected component $\Delta_i$ of $\Delta$ must be $g$-stable because $Z$ is assumed to have no Du Val singular points.

   \medskip

   Step 1. $\Delta_i=A_n$.

   Suppose $\Delta_i=D_n$ or $E_n$. Let $C$ be the center of $\Delta_i$. Then $C$ meets exactly $3$ smooth rational curves in $\Delta_i$, say $C_1,C_2,C_3$. By uniqueness, $C$ is $g$-stable, and $g(\{C_1,C_2,C_3\})=\{C_1,C_2,C_3\}$.

   If every $C_j$ is $g$-stable, then $C$ has at least $3$ $g$-fixed points, and it is $g$-fixed. Hence, $C_j$ are not $g$-fixed. On the other hand, each $C_j$ contains two $g$-fixed points, and one of them is not in $C$.  There would be another $g$-fixed curve $C_j'$ in $\Delta_i$ which intersects $C_j$, $j=1,2,3$, a contradiction. Suppose $C_1$ is not $g$-stable, say $g(C_1)=C_2$. Then $g(C_2)=C_1$ and $C$ is not $g$-fixed. Since $C_3$ is $g$-stable, by Lemma~\ref{l2.12} it is also $g$-fixed. However, one of the two $g$-fixed points on $C$ is not contained in $C_3$, so $C$ should intersect with another $g$-fixed curve in $\Delta_i$, a contradiction again.

   Therefore, we can express $\Delta_i=A_n$ as a linear chain of smooth rational curves: $C_1-C_2-\cdots-C_n$.

   \medskip

   Step 2. Each $C_j$ is $g$-stable.

   Suppose $g(C_1)\neq C_1$.  Then $g(C_1)=C_n$, and consequently $g(C_j)=C_{n-j}$ for all $j$. There are two cases.

i) If $m=2k$, let $\{P\}=C_k\cap C_{k+1}$, then $P$ would be an
isolated $g$-fixed point, absurd.

ii) If $m=2k-1$, then $C_k$ is $g$-stable, and there would be a $g$-fixed curve which intersects $C_k$. But $\Delta_i$ contains no $g$-fixed curve, a contradiction.

Therefore, $g(C_1)=C_1$ and it follows that each $C_j$ is $g$-stable.

\medskip

Step 3. $\Delta_i=A_{2m-1}$.

Note that each $g$-stable but not $g$-fixed curve must intersect
$g$-fixed curves at two points. So $C_1$ is $g$-fixed and $C_2$ is
not. Consequently, each $C_{2j-1}$ is $g$-fixed and $C_{2j}$ is not.
With the same reason, $C_n$ must be $g$-fixed. So $n$ is odd.
Therefore, $\Delta_i=A_n$ has the form
\begin{equation*}
\begin{psmatrix}[rowsep=3mm,colsep=5mm,linewidth=0.5pt, nodesep=1mm]
  f & s & f & s& f & \cdots & f & s & f
  \ncline{1,1}{1,2}
  \ncline{1,2}{1,3}
  \ncline{1,3}{1,4}
  \ncline{1,4}{1,5}
  \ncline{1,5}{1,6}
  \ncline{1,6}{1,7}
  \ncline{1,7}{1,8}
  \ncline{1,8}{1,9}
\end{psmatrix}
\end{equation*}
where ``$f$'' denotes the $g$-fixed curves and ``$s$'' denotes the $g$-stable but not $g$-fixed curves in $\Delta_i$.

\medskip

Step 4. Determine the Dynkin type of $\Delta$.

Decompose $\Delta= \bigoplus_{i=1}^r A_{2n_i-1}$. Recall that every
smooth rational $g$-fixed curve in $S$ is contained in $\Delta$. Let
$N$ be the number of smooth rational $g$-fixed curves in $S$.  Then
$N=\sum_{i=1}^r n_i$ and
\begin{equation*}
  18= \rank \Delta= \sum_{i=1}^r (2n_i-1)= 2N-r.
\end{equation*}
So we have
\begin{equation*}
  N= \dfrac{18+r}{2}>9.
\end{equation*}
Then $N\geq 10$. It follows from Proposition~\ref{p2.8} that $N=10$
and $(S,g)\simeq (S_2,g_2)$. Moreover, $r=2$. This completes the
proof.
\end{proof}

We have the following configurations for $\Delta$:
\begin{equation*}
  A_1\oplus A_{17},\quad A_3\oplus A_{15},\quad A_5\oplus A_{13},\quad A_7\oplus A_{11},\quad A_9\oplus A_9.
\end{equation*}

Similarly as in the case when $I=3$, if $S_2^g\subseteq \Delta$ and
the divisor $\Delta$ can  be obtained from the $24$ smooth rational
curves in $S_2$ (Figure~\ref{fig2.2}) which satisfies the conditions in the proof of Proposition~\ref{p2.13} Step~3, let $S_2\to \bar S$ be the
contraction of $\Delta$, then the automorphism $g_2$ on $S_2$
induces an automorphism on $\bar S$, and $Z:=\bar S/\langle
g_2\rangle$ is a required log Enriques surface of Dynkin's type
$\Delta$.

We can easily verify that these $5$ cases are all realizable (cf.\
Table~\ref{fig2.2}). We have proved Main Theorem (2). By noting the results in Step 2 and 3 in the proof of Proposition~\ref{p2.13}, Main Theorem (5) for case $I=2$ is also proved.

\section*{4.3. Classification When $I=4$}\label{sec2.4.3}

Let $(S,g)$ be a pair of a smooth K3 surface $S$ and an automorphism $g$ of $S$. Assume that $g^4=\id$ and
$g^*\omega_S=i\omega_S$ for a nowhere vanishing holomorphic $2$-form on $S$, where $i=\sqrt{-1}$. Let $P$ be an isolated $g$-fixed point. Then $g^*$ can be written as $\diag(-1,-i)$ near $P$ with appropriate coordinates. Let $C$ be a $g$-fixed irreducible curve and $Q$ a point in $C$. Then $g^*$ can be written as $\diag(1,i)$ near $Q$ with appropriate coordinates.

%Using the same argument as for case $I=3$, every smooth curve in
%$S^h$ is contained in $\Delta$. So $S^h$ consists of smooth rational
%curves and finitely many isolated points. Moreover, $\Delta$ is
%$g$-stable.

Similarly as in the case $I=2$ (Lemma~\ref{l2.12}) or $I=3$
(Lemma~\ref{l2.10}), we can state and prove the following lemma.

\begin{Lemma}[``Four Go'' Lemma]\label{l2.14}
  Let $(S,g)$ be a pair of smooth K3 surface $S$ and an automorphism $g$ of $S$. Assume that $g^4=\id$ and $g^*\omega_S=i\omega_S$.
  \begin{enumerate}[\quad \rm 1)]\itemsep=1mm
    \item Let $C_1-C_2-C_3-C_4$ be a chain of $g$-stable smooth rational curves. Then exactly one of $C_j$ is $g$-fixed, and exactly one of $C_k$ is $g^2$-fixed but not $g$-fixed. Moreover, $\{j,k\}=\{1,3\}$ or $\{2,4\}$.
    \item Let $C$ be a $g$-stable but not $h$-fixed smooth rational curve on $S$.
    Then there exists a unique $g$-fixed curve $D_1$ and a unique $g^2$-fixed but not $g$-fixed curve $D_2$
    such that $C\cdot D_1=C\cdot D_2=1$.
    \item Let $M$ and $N$ be the number of smooth rational curves
    and the number of isolated points in $S^g$, respectively. Then
    $M-2N=4$.
  \end{enumerate}
\end{Lemma}

\begin{proof}
  1) Applying Lemma~\ref{l2.12} to $h:=g^2$, we may assume that
  $C_1,C_3$ are $h$-fixed and $C_2,C_4$ are not. Note that $\{P\}=C_1\cap C_2$ and $\{Q\}=C_2\cap
  C_3$ are $g$-fixed. The action of $g$ on the
  tangent space $T_{C_2,P}$ of $C_2$ at $P$ is the
  multiplicative of $i$ or $-i$, and the action of $g$ on
  $T_{C_2,Q}$ is the multiplicative of $-i$ or $i$, respectively.
  For the first case, $C_1$ is $g$-fixed and $C_3$ not; and
  conversely for the second case.

  2) Let $P$ and $Q$ be the $g$-fixed points on $C$. Then the
  actions of $g$ on $T_{C,P}$ and $T_{C,Q}$ are the multiplication
  of $i$ and $-i$, respectively. So there is a unique $g$-fixed curve
  passing through $P$ and a unique $h$-fixed but not $g$-fixed curve
  passing through $Q$.

  3) We can write
\begin{equation*}
  S^g = \{P_1\}\coprod \cdots \coprod \{P_M\} \coprod C_1 \coprod \cdots \coprod C_N,
\end{equation*}
where $P_j$ are the isolated $g$-fixed points, and $C_k$ are the
smooth irreducible rational  $g$-fixed curves of $S$. Consider the
holomorphic Lefschetz number $L(g)$, which can be evaluated in two
different ways.

Method 1. $L(g)= \displaystyle \sum_{i=0}^2 (-1)^i
\tr(g^*|_{H^i(S,\mathcal O_S)})$ (cf.~\cite[\S3]{atiyah1968segal}).

We see that $H^0(S,\mathcal O_S)\simeq \mathbb C$, $H^1(S, \mathcal O_S)=0$, and by Serre duality
\begin{equation*}
  H^2(S, \mathcal O_S)\simeq H^0(S, \mathcal O_S(K_S))^\vee = H^0(S, \mathcal O_S)^\vee.
\end{equation*}
Then $g^*|_{H^0(S,\mathcal O_S)}=\id$, $g^*|_{H^1(S, \mathcal
O_S)}=0$  and $g^*|_{H^2(S,\mathcal O_S)}= i^{-1}= -i$.

Method 2. $L(g)= \displaystyle \sum_{j=1}^M a(P_j)+ \sum_{k=1}^N
b(C_k)$.
\begin{align*}
  a(P_j): & = \dfrac{1}{\det(1-g^*|_{T_{P_j}})},\\
  b(C_k): & =
  \dfrac{1-\pi(C_k)}{1-\lambda_k^{-1}}-\dfrac{\lambda_k^{-1}}{(1-\lambda_k^{-1})^2}(C_k)^2,
\end{align*}
where $\pi(C_k)$ is the genus and $(C_k)^2$ is the self-intersection
number of $C_k$, and $\lambda_k$ is the eigenvalue of $g^*$ on the
normal bundle of $C_k$ (cf.~\cite[\S4]{atiyah1968singer}).

Recall that $g^*|_{T_{P_j}}= \diag(-1,-i)$. Then
\begin{equation*}
  a(P_j)=\dfrac{1}{(1+1)(1+i)}= \dfrac{1-i}{4}.
\end{equation*}
Since $g^*|_{T_{Q_k}}= \diag(1,i)$ near $Q_k\in C_k$,
$\lambda_k=i^{-1}$ is the eigenvalue of $g^*$ on the normal bundle. So
\begin{equation*}
  b(C_k)  = \dfrac{1-0}{1-i}- \dfrac{i}{(1-i)^2}(-2) = -\dfrac{1-i}{2}.
\end{equation*}
Therefore, $1-i=\dfrac{M}{4}(1-i)- \dfrac{N}{2}(1-i)$; that is,
$M-2N=4$.
\end{proof}

\medskip

Now suppose $I(Z)=4$. Then the associated pair $(S,g)$ satisfies the conditions in Lemma~\ref{l2.9} and \ref{l2.14}. Set $h:=g^2$. First of all, we claim that
\begin{Lemma}\label{l2.15}
  With the notations as in Main Theorem and above, each connected component $\Delta_i$ of $\Delta$ is $h$-stable.
\end{Lemma}

\begin{proof}
  Suppose $\Delta_i$ is not $h$-stable. Then $\Delta_i$, $g(\Delta_i)$, $h(\Delta_i)$
  and $g^3(\Delta_i)$ are distinct components in $\Delta$, and they are contracted to
  Du Val singular points on $\bar S/\langle g\rangle$, a contradiction to our assumption.
\end{proof}

Therefore, applying Proposition~\ref{p2.8} to $(S,h)$ we have
$(S,h)\simeq (S_2,g_2)$, the Shioda-Inose's pair of discriminant 4. From now on,
we set $(S,h)=(S_2,g_2)$. Since is known that $(g_2^*)^2=\id$ on $\Pic(S)$, we can write $g^*|_{\Pic(S)\otimes \mathbb
C}=\diag(I_s,-I_t)$, where $s+t=\rho(S)=20$. Let $x\in T_S$. Suppose $g^*x=\pm x$. Then
\begin{equation*}
  x\cdot \omega_S= g^*(x\cdot \omega_S)= g^*x \cdot g^*\omega_S =
  \pm x\cdot i\omega_S= \pm i (x\cdot \omega_S).
\end{equation*}
It follows that $x\cdot \omega_S=0$. Then $x\in \Pic(S)\cap
T_S=\{0\}$. So $\pm 1$ are not eigenvalues of $g^*|_{T_S\otimes
\mathbb C}$. By Lemma~\ref{l2.9}, we can thus write $g^*|_{T_S\otimes \mathbb C}=
\diag(i,-i)$.

\medskip

\begin{Proposition}\label{p2.16}
  With the notations as in Main Theorem. Suppose $I=4$. Let $h=g^2$. Then $(S,h)\simeq (S_2,g_2)$, the Shioda-Inose's pair of discriminant $4$. Moreover, $\Delta$ is of the
  type $A_1\oplus A_{17}$, $A_5\oplus A_{13}$ or $A_9\oplus A_9$.
\end{Proposition}

\begin{proof}
  We only need to check the second assertion. Let $M$ be the number of isolated $g$-fixed points and $N$ the
  number of smooth irreducible $g$-fixed curves. By
  Lemma~\ref{l2.14}, we have $M-2N=4$.

\medskip

Step 1. $N\leq 4$.

%Step 2. $M=12$, $N=4$ and $K=6$, where $K$ is the number of $h$-fixed but not $g$-fixed smooth rational curves on $S$.

We apply the topological Lefschetz fixed point theorem (cf.\ \cite[Lemma~1.6]{ueno1976remark}),
\begin{equation*}
  \chi_{\top}(S^g)= \sum_{i=0}^4 (-1)^i \tr(g^*|_{H^i(S,\mathbb
  Q)}).
\end{equation*}
The left-hand side is $M+2N=4N+4$, and the right-hand side is
\begin{equation*}
  2+\tr(g^*|_{\Pic(S)\otimes\mathbb C})+ \tr(g^*|_{T_S\otimes\mathbb C})= 2+ s-t.
\end{equation*}
where $g^*|_{\Pic(S)\otimes\mathbb C}= \diag(I_s, -I_t)$. Since
$s+t=\rho(S)=20$, we have
\begin{equation*}
  s=11+2N\quad \textrm{and}\quad t=9-2N.
\end{equation*}
It follows that $N\leq 4$.

\medskip

Step 2. $\Delta= A_{2m-1}\oplus A_{2n-1}$, where $m+n=10$.

This follows immediately from Proposition~\ref{p2.13}.

\medskip

Step 3. $\Delta\neq A_3\oplus A_{15}$ and $\Delta\neq A_7\oplus
A_{11}$. So Proposition~\ref{p2.16} will follow.

i) Suppose $\Delta=A_3\oplus A_{15}$. Denote $A_3=C_1-C_2-C_3$ and
$A_{15}=D_1-D_2-\cdots-D_{15}$. Then it follows from the proof of
Proposition~\ref{p2.13} that all $C_i$ and $D_j$ are $h$-stable, and
from which
\begin{equation*}
  C_1, C_3, D_1, D_3, D_5, D_7, D_9, D_{11}, D_{13}, D_{15}
\end{equation*}
are $h$-fixed and others are not. Clearly each connected component is $g$-stable, and
$\Aut(\Delta)=(\mathbb Z/2\mathbb Z)\oplus (\mathbb Z/2\mathbb Z)$. Note that $g(C_1)=C_1$ or $C_3$. For each case $C_2$ is $g$-stable
but not $h$-fixed. By Lemma~\ref{l2.14}, $C_2$ intersects with a
unique $g$-fixed curve. Then $C_1$ or $C_3$ is $g$-stable, and
therefore all $C_i$ are $g$-stable. Similarly, by noting that $D_8$
is $g$-stable but not $h$-fixed, we see that all $D_j$ are
$g$-stable. By Lemma~\ref{l2.14} again, $C_1, D_1, D_5, D_9, D_{13}$
must be $g$-fixed. But this contradicts $N\leq 4$.

ii) Suppose $\Delta=A_7\oplus A_{11}$. Denote
$A_7=C_1-C_2-\cdots-C_7$ and $A_{11}=D_1-D_2-\cdots-D_{11}$. Then
using the same argument as for $A_3\oplus A_{15}$, we can show that
$C_i$ and $D_j$ are $g$-stable for all $i,j$, and therefore $C_1,
C_5$, $D_1, D_5, D_9$ are $g$-fixed. This contradicts $N\leq 4$
again.
\end{proof}

\begin{proof}[Proof of Main Theorem (4)]

It remains to show that $A_1\oplus A_{17}$, $A_5\oplus A_{13}$ and
$A_9\oplus A_9$ are realizable.

Let $g_4$ be the automorphism of $S_2$ induced by the action
$\diag(i,1)$ on $E_{\zeta_4}^2$. Then $g_4^2=g_2$ as in Definition~\ref{d2.6}. From the construction of the
$24$ rational curves in $S_2$ (Figure~\ref{fig2.2}), we see that

\begin{enumerate}[\quad I)]\itemsep=2mm
  \item $4$ curves are $g_4$-fixed, say $F_1, F_2$ and $G_1, G_3$;
  \item $6$ curves are $g_2$-fixed but not $g_4$-fixed, say $F_2,
  G_2, H_{11}, H_{13}, H_{31}, H_{33}$;
  \item $g_4(H_{22})=H_{22}'$ and $g_4(H_{22}')=H_{22}$;
  \item the remaining $12$ curves are $g_4$-stable, but not
  $g_2$-fixed.
\end{enumerate}

Let $g:=g_4$ and $h:=g^2$. Then $\Delta$ contains exactly $4$ $g$-fixed curves (i.e., $N=4$), and $6$ $h$-fixed but not $g$-fixed curves. Consider the following three possible types of $\Delta$.

i) $A_1\oplus A_{17}$.

Since $A_1$ contains at most 1 $g$-fixed curve, $A_{17}$ must contain
at least $3$ $g$-fixed curves. Then every curve in $A_{17}$ is
$g$-stable. Moreover, it contains $9$ $h$-fixed curves.  Noting that $\Delta$ has exactly $4$ $g$-fixed curves, we see that $C_3,C_7,C_{11},C_{15}$ are the $g$-fixed curves and $C_1,C_5,C_9,C_{13},C_{17},A_1$ are the $h$-fixed but not $g$-fixed curves.

ii) $A_5\oplus A_{13}$.

Since $A_5$ contains at most $2$ $g$-fixed curves, $A_{13}$ has a
$g$-fixed curve. So every curve in $A_{13}$ is $g$-stable. We write
\begin{align*}
  A_5 & = C_1-C_2-C_3-C_4-C_5, \\
  A_{13} & = D_1-D_2-D_3-\cdots -D_{13}.
\end{align*}
If $C_1$ is not $g$-stable, then only $C_3$ in $A_5$ is $h$-fixed. Note that it is not $g$-fixed. Then $A_{13}$ shall contain $4$ $g$-fixed curves: $D_1,D_5,D_9,D_{13}$. However, $\Delta$ would have only $5$ $h$-fixed but not $g$-fixed curves $D_3,D_7, D_{11}, D_{15}, C_3$, a contradiction. Therefore, every curve in $A_5$ is $g$-stable. Then $A_5$ contains at least
$1$ $g$-fixed curve, and $A_{13}$ contains at most $3$ $g$-fixed curves. It follows that exactly $4$ curves $C_3,D_3,D_7,D_{11}$ in $\Delta$ are $g$-fixed.

iii) $A_9\oplus A_9$.

We call the second $A_9$ as $A_9'$. If $A_9$ is not $g$-stable, then $g(A_9)=A_9'$ and $g(A_9')=A_9$.
There would be no $g$-fixed curve in $\Delta$, absurd. So both $A_9$ and
$A_9'$ are $g$-stable. Since $A_9$ contains at most $3$ $g$-fixed curves, $A_9'$
contains at least $1$ $g$-fixed curve. Hence every curve in $A_9'$
is $g$-stable. Similarly, every curve in $A_9$ is $g$-stable. On the other hand, $A_9$ should contain at least $2$ $g$-fixed
curves, so does $A_9'$. If we write
\begin{align*}
  A_9 & =C_1- C_2-C_3-\cdots-C_9,\\
  A_9'& =D_1-D_2-D_3-\cdots-D_9,
\end{align*}
then exactly $C_3,C_7,D_3$ and $D_7$ are $g$-fixed.

\medskip

Since we have determined the action of $g$ on $\Delta$ and these
$\Delta$ can be obtained from the 22 $g$-stable rational curves in
$S_2$ (Figure~\ref{fig2.2}), they are all realizable. The dual
graphs are given in Table~\ref{table2.2} (1), (3) and (5).
\end{proof}

Note that in the proof of above, we showed that for each of the
every cases, every irreducible curve in $\Delta$ is $g$-stable.

\section*{4.4. Impossibility of $I=6$}\label{2.4.4}

In order to complete the proof of Main Theorem, in this section we will explore the method used in
\cite[Proposition~2.12, Lemma~2.13]{oguiso1999complete} to prove the following.

\begin{Proposition}\label{p2.17}
  With the notations in Main Theorem, $I\neq 6$.
\end{Proposition}

\begin{proof}
  We assume that there is a log Enriques surface $Z$ of rank 18 without Du Val singularities. Let $(S,g)$ be the associated pair. Let $P$ be an isolated $g$-fixed point. Then $g^*$ can be written as either
\begin{enumerate}[\quad \rm i)]\itemsep=1mm
  \item $\diag(\zeta_6^2,\zeta_6^5)$, or
  \item $\diag(\zeta_6^3,\zeta_6^4)$
\end{enumerate}
with appropriate coordinates around $P$.

\medskip

Step 1. There are even number of isolated $g$-fixed points of the second type.

Suppose $g^*=\diag(\zeta_6^2,\zeta_6^5)$ near $P$. Then
$(g^2)^*=(\zeta_6^4,\zeta_6^4)$ near $P$. It follows that $P$ is an
isolated $g^2$-fixed point. Suppose $g^*=\diag(\zeta_6^3,\zeta_6^4)$ near $P$. Then
$(g^2)^*=\diag(1,\zeta_6^2)$,  and there exists a unique smooth
rational $g^2$-fixed curve $C$ passing through $P$. Since $S^{g^2}$
is smooth, $C$ is $g$-stable but not $g$-fixed. Let $Q$ be the other
$g$-fixed point on $C$. Since $Q$ is not an isolated $g^2$-fixed
point, it is also an isolated $g$-fixed point of the second type. Therefore, the $g$-fixed points of the second type come in pairs.
There are even number of such points.

\medskip

Step 2. The number of isolated $g$-fixed points of the first type equals that of the second type.

Let $P$ be an isolated $g$-fixed point. Since $S^g\subseteq
S^{g^3}$,  a disjoint union of smooth rational curves, there is a
unique $g^3$-fixed curve $C$ passing through $P$. Hence, $C$ is
$g$-stable but not $g$-fixed, and it contains exactly $2$ $g$-fixed
points. Note that if $P$ is of the first type
$\diag(\zeta_6^2,\zeta_6^5)$, then $g^*|_{T_{C,P}}=\zeta_6^2$; if
$P$ is of the second type $\diag(\zeta_6^3, \zeta_6^4)$, then
$g^*|_{T_{C,P}}=\zeta_6^4$. So the other isolated $g$-fixed point on
$C$ is of different type of $P$. Therefore, there is a one-to-one correspondence between the set of
$g$-fixed points of the first type and that of the second type. Step
2 is proved.

\medskip

Now we can set $P_1,\ldots,P_{2\ell}$ and $Q_1,\ldots,Q_{2\ell}$ to
be the isolated $S^g$-fixed points of type
$\diag(\zeta_6^2,\zeta_6^5)$ and of type
$\diag(\zeta_6^3,\zeta_6^4)$, respectively. Suppose there are $c$ rational smooth $g$-fixed curves, say
$C_1,\ldots,C_c$. We claim that

\medskip

Step 3. $\ell=c+1$.

Similarly as in the proof of Lemma~\ref{l2.14}, we use the
holomorphic Lefschetz fixed point formula
\begin{equation*}
  L(g)= \sum_{i=0}^2 (-1)^i \tr(g^*|_{H^i(S,\mathcal O_S)}) = \sum_{i=1}^{2\ell} a(P_i)+ \sum_{i=1}^{2\ell} a(Q_i)+ \sum_{i=1}^c b(C_i).
\end{equation*}
We can compute that
\begin{equation*}
  \sum_{i=0}^2(-1)^i \tr(g^*|_{H^i(S,\mathcal O_S)})=1+0+\dfrac{1}{\zeta_6}= \dfrac{3-i\sqrt 3}{2}.
\end{equation*}
\begin{align*}
  a(P_i) & = \dfrac{1}{\det(1-g^*|_{T_{P_i}})}= \dfrac{1}{(1-\zeta_6^2)(1-\zeta_6^5)}= \dfrac{3-i\sqrt 3}{6}, \\
  a(Q_i) & = \dfrac{1}{\det(1-g^*|_{T_{Q_i}})}= \dfrac{1}{(1-\zeta_6^3)(1-\zeta_6^4)}= \dfrac{3-i\sqrt 3}{12},\\
  b(C_i) & = \dfrac{1-\pi(C_i)}{1-\zeta_6}- \dfrac{\zeta_6\, C_i^2}{(1-\zeta_6)^2}= -\dfrac{3-i\sqrt 3}{2}.
\end{align*}
Therefore, $\ell=c+1$.

\medskip

Step 4. Determine $S^{g^2}$.

If $P$ is a $g^2$-fixed but not $g$-fixed point, then so is $g(P)$. Therefore, there are even number of $g^2$-fixed but not $g$-fixed points. If $C$ is a rational smooth irreducible $g^2$-fixed curve which does not contain any $g$-fixed point, so is $g(C)$. Hence, there are even number of such curves.

Suppose the isolated $g^2$-fixed points are $P_1,\ldots,P_{2c+2}, R_1,\ldots, R_{2k}$, and the smooth rational $g^2$-fixed curves are $C_1, \ldots,C_c, D_1,\ldots,D_{c+1},\ldots,F_1,\ldots, F_{2p}$, where $R_i$ is  not $g$-fixed, $Q_{2i-1},Q_{2i}\in D_i$, and $F_i$ does not contain at $g$-fixed point. Then apply Lemma~\ref{l2.10} to $(S,g^2)$, we obtain
\begin{equation*}
  (2c+2+2k)-(c+c+1+2p)=3,
\end{equation*}
which implies $k=p+1$.

\medskip

Step 5. Determine $S^{g^3}$.

We note $g^3$ is a non-symplectic involution on $S$, and so there is no isolated $g^3$-fixed point. If $G$ is a $g^3$-fixed curve which does not contain any $g$-fixed point, then so are $g(G)$ and $g^2(G)$. Therefore, the smooth rational $g^3$-fixed curves are $C_1,\ldots,C_c$, $E_1,\ldots,E_{2c+2}$, $G_1,\ldots,G_{3q}$, where $P_i,Q_i\in E_i$ and $G_i$ does not contain any $g$-fixed point.

\medskip

Step 6. $c+p+q\leq 2$.

Since $\ord(g)=6$, we can write
\begin{equation*}
  g^*|_{H^2(S,\mathbb Q)}= \diag(I_\alpha,-I_\beta, \zeta_6^2 I_\gamma, \bar\zeta_6^2 I_\gamma, \zeta_6 I_{1+\delta}, \bar\zeta_6 I_{1+\delta}),
\end{equation*}
where $\alpha,\beta,\gamma,\delta\geq 0$. Let $j=1$ in the topological Lefschetz fixed point formula
\begin{equation*}
  \chi_{\top}(S^{g^j})= \sum_{i=0}^4 (-1)^i \tr((g^j)^*|_{H^i(S,\mathbb Q)}).
\end{equation*}
We have
\begin{equation*}
  (2c+2)+(2c+2)+2\cdot c=2+\alpha-\beta-\gamma+(\delta+1).
\end{equation*}

$(g^2)^*|_{H^2(S,\mathbb Q)}= \diag(I_{\alpha+\beta},
\zeta_6^2I_{\gamma+\delta+1}, \bar\zeta_6^2 I_{\gamma+\delta+1})$.
Then for $j=2$ we have
\begin{equation*}
  (2c+2)+(2p+2)+2[c+(c+1)+2p]= 2+(\alpha+\beta)-(\gamma+\delta+1).
\end{equation*}

$(g^3)^*|_{H^2(S,\mathbb Q)}= \diag(I_{\alpha+2\gamma}, -I_{\beta+2+2\delta})$. Then for $j=3$ we have
\begin{equation*}
  2[c+(2c+2)+3q]=2+(\alpha+2\gamma)-(\beta+2+2\delta).
\end{equation*}
We also note that
\begin{equation*}
  \alpha+\beta+2\gamma+2(1+\delta)=\dim H^2(S,\mathbb Q)=22.
\end{equation*}
It can be solved that $\delta=-c-p-q+2$. In particular, $c+p+q\leq 2$.

\medskip

Step 7. Determine the possible types of $\Delta$.

Let $\Delta_i$ be a connected component of $\Delta$. Then $\Delta_i$
is either $g^3$-stable or $g^2$-stable, otherwise $g^k(\Delta_i)$,
$k=0,\ldots,5$, would be contracted to a single Du Val singular
point in $\bar S/\langle g\rangle$, which should not exist by
assumption.

Suppose $\Delta_i$, $i=1,\ldots,m$, are the $g^3$-stable connected
components of $\Delta$. Since $(g^3)^*\omega_S= -\omega_S$, using
the same argument as for $I=2$, we see that $\Delta_i= A_{2m_i-1}$
for some $m_i$, which contains exactly $m_i$ smooth rational
$g^3$-fixed curves. On the other hand, by computation in Step 4,
there are $c+(2c+c)+3q=3(c+q)+2$ $g$-fixed curves. Therefore,
\begin{equation*}
  \sum_{i=1}^n \rank \Delta_i = \sum_{i=1}^m (2m_i-1)= 6(c+q)+4-m.
\end{equation*}
Since $\ell=c+1>0$, $S^g\neq \emptyset$. We see that $m\geq 1$.

Suppose $\Delta'_j$, $j=1,\ldots,n$, are the $g^2$-stable but not
$g$-stable connected components of $\Delta$.  Since
$(g^2)^*\omega_S= \zeta_3\omega_S$, using the same argument as for
$I=3$, we see that each $\Delta_j'$ has Dynkin type $A$ or $D$.

Since each $\Delta_j'$ contains at least one $g^2$-fixed curve and
$F_1,\ldots,F_{2p}$ are the only $g^2$-fixed curves in $\Delta'_j$,
we have $n\leq 2p$. On the other hand, from the proof of
Proposition~\ref{p2.11} Step 4, if $\rank \Delta_j'=\alpha_j$, then
$\Delta_j$ contains at least $\lceil(\alpha_j-1)/3\rceil$ smooth
$g^2$-fixed curves. We have an estimation
\begin{equation*}
  2p\geq \sum_{j=1}^n \lceil(\alpha_j-1)/3\rceil \geq \sum_{j=1}^n
  (\alpha_j-1)/3.
\end{equation*}
That is,
\begin{equation*}
  \sum_{j=1}^n \rank \Delta_j' \leq 6p+n.
\end{equation*}
Note that $\Delta_j'$ is not $g^3$-stable, otherwise it would also
be $g$-stable. So $\Delta_j'$ and $g^3(\Delta_j')$ are disjoint
connected components in $\Delta$. In particular, $n$ is even. It follows from $\rank \Delta=18$ that
\begin{align*}
  18 & \leq 6(c+q)+4-m+6p+n =6(c+p+q)+4-m+n \\
  & \leq 6\cdot 2+4-m+n =16-m+n\\
  & \leq 16-1+n = 15+n\\
  & \leq 15+2p.
\end{align*}
Then $p\geq 2$ and it follows from $c+p+q\leq 2$ that $p=2$ and
$c=q=0$. So $\Delta$ has no $g$-fixed curve. Since $n$ is even, $n=4$ and $m=1$ or $2$. We are left to show that these two cases are impossible.

\medskip

Recall that $\Delta_i$ has the form $A_{2m_i-1}$ and contains
exactly $m_i$ $g^3$-fixed curves, and the $2$ irreducible
$g^3$-fixed curves are contained in $\coprod_{i=1}^m \Delta_i$. We
have $\sum_{i=1}^m m_i=2$.

If $m=1$, then $m_1=2$ and $\Delta_1=A_3$. However, this would imply
that $\sum_{j=1}^4 \rank \Delta_j'= 15$, which needs to be even. If $m=2$, then $m_1=m_2=1$ and $\Delta_1=\Delta_2=A_1$. They are
$g^3$-fixed. On the other hand, note that $\ord(g^2)=3$. By
considering the $g^2$-action on $\Delta$, we see that $\Delta_1$ and
$\Delta_2$ are also $g^2$-fixed. It follows that $\Delta_1$ and
$\Delta_2$ $g$-fixed, which contradicts our computation that there
is no $g$-fixed curve.
\end{proof}

This complete the proof of Proposition~\ref{p2.17} and also Main Theorem (1).

\section[List of Dynkin's Types]{The List of Dynkin's Types of $\Delta$}\label{sec2.5}

\begin{table}[h]
  \caption{$I=3$}\label{table2.1}
\end{table}
``$f$'' denotes the $g$-fixed curve and $s$ denotes the $g$-stable
but not $g$-fixed curve. We use the same labeling for curves as in
Figure~\ref{fig2.1}.
\begin{center}
  (A) Realizable Cases.
\end{center}
{\footnotesize

Case I: $A_{18}$: $s-f-s-s-f-s-s-f-s-s-f-s-s-f-s-s-f-s$

$E_{33}-G_3-E_{13}-E_{13}'-F_1-E_{11}'-E_{11}- G_1-E_{31}- E_{31}'-F_3-E_{32}'-E_{32}-G_2-E_{22}-E_{22}'-F_2-E_{21}'$.

Case II: $D_{18}$: $\displaystyle {s\atop
s}\!>f-s-s-f-s-s-f-s-s-f-s-s-f-s-s-f$

$\displaystyle {{E_{11}'}\atop{E_{12}'}}\!>F_1-E_{13}'-E_{13}-G_3- E_{33}-E_{33}' -F_3 -E_{31}'-E_{31}-G_1-E_{21}-E_{21}'-F_2- E_{22}'-E_{22}-G_2$

Case III: $A_{3m}\oplus A_{3n}$, where $m+n=6$, $1\leq m\leq n\leq 5$.

(1) $A_3\oplus A_{15}$: $s-f-s$, $s-f-s-s-f-s-s-f-s-s-f-s-s-f-s$

$E_{11}'-F_1-E_{12}'$

$E_{13}-G_3- E_{33}-E_{33}'- F_3- E_{31}'- E_{31}- G_1- E_{21}- E_{21}'- F_2- E_{22}'- E_{22}- G_2 - E_{32}$

(2) $A_6\oplus A_{12}$: $s-f-s-s-f-s$, $s-f-s-s-f-s-s-f-s-s-f-s$

$E_{21}- G_1- E_{11}- E_{11}'- F_1- E_{12}'$

$E_{13}-G_3- E_{23}-E_{23}'- F_2- E_{22}'- E_{22}- G_2- E_{32}- E_{32}'- F_3- E_{33}'$

(3) $A_9\oplus A_9$: $s-f-s-s-f-s-s-f-s$, $s-f-s-s-f-s-s-f-s$

$E_{11}'- F_1- E_{12}'- E_{12}- G_2- E_{22}- E_{22}'- F_2 - E_{23}'$

$E_{13}- G_3- E_{33}-E_{33}'- F_3- E_{31}'- E_{31}-G_1 - E_{21}$

Case IV: $D_{3m}\oplus A_{3n}$, where $m+n=6$.

(1) $D_6 \oplus A_{12}$: $\displaystyle {s\atop s}\!>f-s-s-f$,
$s-f-s-s-f-s-s-f-s-s-f-s$

$\displaystyle {{E_{11}'}\atop {E_{12}'}}\!>F_1- E_{13}'- E_{13}- G_3$

$E_{33}'- F_3-E_{32}'-E_{32}- G_2- E_{22}-E_{22}'-F_2 - E_{21}'-E_{21}-G_1- E_{31}$

(2) $D_9\oplus A_9$: $\displaystyle {s\atop s}\!>f-s-s-f-s-s-f$, $s-f-s-s-f-s-s-f-s$

$\displaystyle {{E_{11}'}\atop {E_{12}'}}\!>F_1- E_{13}'- E_{13}-G_3- E_{23}- E_{23}'- F_2$

$E_{22}-G_2- E_{32}- E_{32}'- F_3- E_{31}'- E_{31}- G_1- E_{21}$

(3) $D_{12}\oplus A_6$: $\displaystyle {s\atop
s}\!>f-s-s-f-s-s-f-s-s-f$, $s-f-s-s-f-s$

$\displaystyle {{E_{11}'}\atop {E_{12}'}}\!> F_1- E_{13}'-E_{13}- G_3- E_{23}- E_{23}'- F_2- E_{22}'- E_{22}- G_2$

$E_{33}'- F_3-E_{31}'-E_{31}-G_1-E_{21}$

(4) $D_{15}\oplus A_3$: $\displaystyle {s\atop
s}\!>f-s-s-f-s-s-f-s-s-f-s-s-f$, $s-f-s$

$\displaystyle {{E_{11}'}\atop {E_{12}'}}\!>F_1-E_{13}'-E_{13}- G_3- E_{23}- E_{23}'- F_2- E_{21}'-E_{21}- G_1- E_{31}- E_{31}'-F_3$

$E_{22}-G_2- E_{32}$

Case V: $D_{3m}\oplus D_{3n}$, where $m+n=6$, $2\leq m\leq n\leq 4$.

(1) $D_6\oplus D_{12}$: $\displaystyle {s\atop s}\!>f-s-s-f$,
$\displaystyle {s\atop s}\!>f-s-s-f-s-s-f-s-s-f$.

$\displaystyle {{E_{11}'}\atop {E_{12}'}}\!> F_1- E_{13}'-E_{13}-G_3$

$\displaystyle {{E_{33}'}\atop {E_{32}'}}\!>F_3- E_{31}'-E_{31}-G_1-E_{21}- E_{21}'- F_2- E_{22}'-E_{22}-G_2$

Case VI: $D_{3n+1}\oplus A_{3m-1}$, $m+n=6$, $1\leq m,n\leq 5$.

(1) $D_4\oplus A_{14}$: $\displaystyle {s\atop s}\!>f-s$,
$f-s-s-f-s-s-f-s-s-f-s-s-f-s$

$\displaystyle {{E_{11}'}\atop {E_{12}'}}\!> F_1- E_{13}'$

$G_3-E_{23}-E_{23}'- F_2- E_{21}'-E_{21}- G_1- E_{31}- E_{31}'- F_3- E_{32}' -E_{32}- G_2- E_{22}$

(2) $D_7\oplus A_{11}$: $\displaystyle {s\atop s}\!>f-s-s-f-s$,
$f-s-s-f-s-s-f-s-s-f-s$

$\displaystyle {{E_{11}'}\atop {E_{12}'}}\!>F_1-E_{13}'-E_{13}- G_3-E_{23}$

$G_2- E_{22}- E_{22}'- F_2- E_{21}'-E_{21}- G_1- E_{31}- E_{31}'-F_3-E_{33}'$

(3) $D_{10}\oplus A_8$: $\displaystyle {s\atop
s}\!>f-s-s-f-s-s-f-s$, $f-s-s-f-s-s-f-s$

$\displaystyle {{E_{11}'}\atop {E_{12}'}}\!>F_1- E_{13}'-E_{13}- G_3- E_{23}- E_{23}'- F_2- E_{21}'$

$G_1-E_{31}- E_{31}'- F_3- E_{32}'- E_{32}- G_2- E_{22}$

(4) $D_{13}\oplus A_5$: $\displaystyle {s\atop
s}\!>f-s-s-f-s-s-f-s-s-f-s$, $f-s-s-f-s$

$\displaystyle {{E_{11}'}\atop {E_{12}'}}\!> F_1-E_{13}'-E_{13} - G_3- E_{33}-E_{33}'-F_3-E_{31}'-E_{31}- G_1-E_{21}$

$G_2-E_{22}-E_{22}'-F_2-E_{23}'$

(5) $D_{16}\oplus A_2$: $\displaystyle {s\atop
s}\!>f-s-s-f-s-s-f-s-s-f-s$, $f-s$

$\displaystyle {{E_{11}'}\atop {E_{12}'}}\!>F_1- E_{13}'-E_{13}-G_3- E_{23}-E_{23}'- F_2-E_{21}'-E_{21}-G_1- E_{31}-
E_{31}'-F_3- E_{32}'$

Case VII: $A_{3m}\oplus A_{3n}\oplus A_{3r}$, $m+n+r=6$, $1\leq m\leq n\leq r\leq 4$.

(1) $A_3\oplus A_3\oplus A_{12}$: $s-f-s$, $s-f-s$,
$s-f-s-s-f-s-s-f-s-s-f-s$

$E_{13}-G_3- E_{23}$

$E_{32}'-F_3-E_{33}'$

$E_{11}'-F_1-E_{12}'-E_{12}-G_2-E_{22}-E_{22}'-F_2- E_{21}'-E_{21}- G_1-E_{31}$

(2) $A_3\oplus A_6\oplus A_9$: $s-f-s$, $s-f-s-s-f-s$,
$s-f-s-s-f-s-s-f-s$

$E_{13}-G_3-E_{33}$

$E_{21}-G_1-E_{31}-E_{31}'-F_3-E_{32}'$

$E_{11}'-F_1-E_{12}'-E_{12}-G_2-E_{22}-E_{22}'-F_2-E_{23}'$

(3) $A_6\oplus A_6 \oplus A_6$: $s-f-s-s-f-s$, $s-f-s-s-f-s$, $s-f-s-s-f-s$

$E_{11}'-F_1-E_{12}'-E_{12}-G_2-E_{22}$

$E_{13}-G_3-E_{33}-E_{33}'-F_3-E_{32}'$

$E_{23}'-F_2-E_{21}'-E_{21}-G_1-E_{31}$

Case VIII: $D_6\oplus D_6\oplus D_6$: $\displaystyle {s\atop
s}\!>f-s-s-f$, $\displaystyle {s\atop s}\!>f-s-s-f$, $\displaystyle
{s\atop s}\!>f-s-s-f$

$\displaystyle {{E_{11}'}\atop {E_{12}'}}\!>F_1-E_{13}'-E_{13}-G_3$

$\displaystyle {{E_{21}'}\atop {E_{23}'}}\!>F_2-E_{22}'-E_{22}-G_2$

$\displaystyle {{E_{32}'}\atop {E_{33}'}}\!>F_3-E_{31}'-E_{31}-G_1$

Case X: $A_{3m}\oplus A_{3n}\oplus D_{3r}$, where $m+n+r=6$, $m\leq n$.

(1) $A_3\oplus A_3\oplus D_{12}$: $s-f-s$, $s-f-s$, $\displaystyle
{s\atop s}\!>f-s-s-f-s-s-f-s-s-f$

$E_{22}-G_2-E_{32}$

$E_{31}'-F_3-E_{33}'$

$\displaystyle {{E_{11}'}\atop {E_{12}'}}\!>F_1-E_{13}'-E_{13}-G_3- E_{23}-E_{23}'-F_2- E_{21}'- E_{21}- G_1$

(2) $A_3\oplus A_6\oplus D_9$: $s-f-s$, $s-f-s-s-f-s$, $\displaystyle {s\atop s}\!>f-s-s-f-s-s-f$

$E_{22}-G_2-E_{32}$

$E_{21}-G_1-E_{31}-E_{31}'-F_3-E_{33}'$

$\displaystyle {{E_{11}'}\atop {E_{12}'}}\!>F_1-E_{13}'-E_{13}-G_3-E_{23}-E_{23}'-F_2$

(3) $A_3\oplus A_9\oplus D_6$: $s-f-s$, $s-f-s-s-f-s-s-f-s$, $\displaystyle {s\atop s}\!>f-s-s-f$.

$E_{22}-G_2-E_{32}$

$E_{23}'-F_2-E_{21}'-E_{21}-G_1-E_{31}-E_{31}'-F_3- E_{33}'$

$\displaystyle {{E_{11}'}\atop {E_{12}'}}\!>F_1-E_{13}'-E_{13}-G_3$

(4) $A_6\oplus A_6\oplus D_6$: $s-f-s-s-f-s$, $s-f-s-s-f-s$, $\displaystyle {s\atop s}\!>f-s-f-s$

$E_{22}-G_2-E_{32}-E_{32}'-F_3-E_{33}'$

$E_{23}'-F_2-E_{21}'-E_{21}-G_1-E_{31}$

$\displaystyle {{E_{11}'}\atop {E_{12}'}}\!>F_1-E_{13}'-E_{13}-G_3$

Case XI: $D_{3m+1}\oplus A_{3n}\oplus A_{3r-1}$, where $m+n+r=6$.

(1) $D_4\oplus A_3\oplus A_{11}$: $\displaystyle {s\atop s}\!>f-s$,
$s-f-s$, $f-s-s-f-s-s-f-s-s-f-s$

$\displaystyle {{E_{11}'}\atop {E_{12}'}}\!>F_1-E_{13}'$

$E_{21}-G_1-E_{31}$

$F_3-E_{32}'-E_{32}-G_2-E_{22}- E_{22}'- F_2- E_{23}'-E_{23} -G_3-E_{33}$

(2) $D_4\oplus A_6\oplus A_8$: $\displaystyle {s\atop s}\!>f-s$, $s-f-s-s-f-s$, $f-s-s-f-s-s-f-s$

$\displaystyle {{E_{11}'}\atop {E_{12}'}}\!>F_1-E_{13}'$

$E_{21}-G_1-E_{31}-E_{31}'-F_3-E_{32}'$

$G_2-E_{22}-E_{22}'-F_2- E_{23}'-E_{23}-G_3-E_{33}$

(3) $D_4\oplus A_9\oplus A_5$: $\displaystyle {s\atop s}\!>f-s$, $s-f-s-s-f-s-s-f-s$, $f-s-s-f-s$

$\displaystyle {{E_{11}'}\atop {E_{12}'}}\!>F_1-E_{13}'$

$E_{21}-G_1- E_{31}-E_{31}'-F_3- E_{33}'-E_{33}- G_3- E_{23}$

$F_2-E_{22}'-E_{22}-G_2-E_{32}$

(4) $D_4\oplus A_{12}\oplus A_2$: $\displaystyle {s\atop s}\!>f-s$,
$s-f-s-s-f-s-s-f-s-s-f-s$, $f-s$

$\displaystyle {{E_{11}'}\atop {E_{12}'}}\!>F_1-E_{13}'$

$E_{21}-G_1-E_{31}-E_{31}'-F_3- E_{32}'-E_{32}- G_2- E_{22}- E_{22}'- F_2- E_{23}'$

$G_3-E_{33}$

(5) $D_7\oplus A_3\oplus A_8$: $\displaystyle {s\atop s}\!>f-s-s-f-s$, $s-f-s$, $f-s-s-f-s-s-f-s$

$\displaystyle {{E_{11}'}\atop {E_{12}'}}\!>F_1- E_{13}'-E_{13}-G_3- E_{33}$

$E_{22}-G_2-E_{32}$

$F_3-E_{31}'-E_{31}- G_1- E_{21}-E_{21}'-F_2- E_{23}'$

(6) $D_7\oplus A_6\oplus A_5$: $\displaystyle {s\atop s}\!>f-s-s-f-s$, $s-f-s-s-f-s$, $f-s-s-f-s$

$\displaystyle {{E_{11}'}\atop {E_{12}'}}\!>F_1- E_{13}'-E_{13}-G_3- E_{33}$

$E_{23}'-F_2- E_{22}'-E_{22}- G_2- E_{32}$

$F_3-E_{31}'-E_{31}-G_1- E_{21}$

(7) $D_7\oplus A_9\oplus A_2$: $\displaystyle {s\atop s}\!>f-s-s-f-s$, $s-f-s-s-f-s-s-f-s$, $f-s$

$\displaystyle {{E_{11}'}\atop {E_{12}'}}\!>F_1- E_{13}'-E_{13}-G_3-E_{33}$

$E_{23}'-F_2-E_{21}'-E_{21}-G_1- E_{31}- E_{31}'- F_3- E_{32}'$

$G_2- E_{22}$

(8) $D_{10}\oplus A_3\oplus A_5$: $\displaystyle {s\atop s}\!>f-s-s-f-s-s-f-s$, $s-f-s$, $f-s-s-f-s$

$\displaystyle {{E_{11}'}\atop {E_{12}'}}\!>F_1- E_{13}'-E_{13}-G_3-E_{33}-E_{33}'- F_3- E_{31}'$

$E_{22}- G_2- E_{32}$

$G_1-E_{21}-E_{21}'- F_2- E_{23}'$

(9) $D_{10}\oplus A_6\oplus A_2$: $\displaystyle {s\atop s}\!>f-s-s-f-s-s-f-s$, $s-f-s-s-f-s$, $f-s$

$\displaystyle {{E_{11}'}\atop {E_{12}'}}\!>F_1-E_{13}'- E_{13}- G_3- E_{33}- E_{33}'- F_3- E_{31}'$

$E_{23}'-F_2- E_{22}'-E_{22}- G_2-E_{32}$

$G_1- E_{21}$

(10) $D_{13}\oplus A_3\oplus A_2$: $\displaystyle {s\atop
s}\!>f-s-s-f-s-s-f-s$, $s-f-s$, $f-s$

$\displaystyle {{E_{11}'}\atop {E_{12}'}}\!>F_1-E_{13}'-E_{13}-G_3- E_{33}- E_{33}'- F_3- E_{31}'-E_{31}- G_1- E_{21}$

$E_{22}- G_2- E_{32}$

$F_2-E_{23}'$

Case XII: $D_{3m+1}\oplus D_{3n+1}\oplus A_{3r-2}$, where $m+n+r=6$, $m\leq n$.

(2) $D_4\oplus D_7\oplus A_7$: $\displaystyle {s\atop s}\!>f-s$, $\displaystyle {s\atop s}\!>f-s-s-f-s$, $f-s-s-f-s-s-f$

$\displaystyle {{E_{11}'}\atop {E_{12}'}}\!>F_1-E_{13}'$

$\displaystyle {{E_{21}'}\atop {E_{22}'}}\!>F_2-E_{23}'-E_{23}-G_3-E_{33}$

$G_1-E_{31}-E_{31}'-F_3-E_{32}'-E_{32}-G_2$

(5) $D_7\oplus D_7\oplus A_4$: $\displaystyle {s\atop s}\!>f-s-s-f-s$, $\displaystyle {s\atop s}\!>f-s-s-f-s$, $f-s-s-f$

$\displaystyle {{E_{12}'}\atop {E_{13}'}}\!>F_1- E_{11}'-E_{11}-G_1-E_{31}$

$\displaystyle {{E_{21}'}\atop {E_{22}'}}\!>F_2- E_{23}'-E_{23}-G_3- E_{33}$

$G_2-E_{32}- E_{32}'- F_3$

(6) $D_7\oplus D_{10} \oplus A_1$: $\displaystyle {s\atop s}\!>f-s-s-f-s$, $\displaystyle {s\atop s}\!>f-s-s-f-s-s-f-s$, $f$

$\displaystyle {{E_{12}'}\atop {E_{13}'}}\!>F_1- E_{11}'-E_{11}- G_1- E_{31}$

$\displaystyle {{E_{21}'}\atop {E_{22}'}}\!> F_2- E_{23}'-E_{23}- G_3-E_{33}- E_{33}'- F_3- E_{32}'$

$G_2$

Case XIII: $D_{3n+1}\oplus D_{3m}\oplus A_{3r-1}$, where $m+n+r=6$, $m\geq 2$.

(3) $D_4\oplus D_{12} \oplus A_2$: $\displaystyle {s\atop s}\!>f-s$,
$\displaystyle {s\atop s}\!>f-s-s-f-s-s-f-s-s-f$, $f-s$

$\displaystyle {{E_{11}'}\atop {E_{12}'}}\!>F_1- E_{13}'$

$\displaystyle {{E_{21}'}\atop {E_{22}'}}\!>F_2- E_{23}'- E_{23}-G_3- E_{33}- E_{33}'- F_3- E_{31}'-E_{31}- G_1$

$G_2-E_{32}$

(4) $D_7\oplus D_6 \oplus A_5$: $\displaystyle {s\atop s}\!>f-s-s-f-s$, $\displaystyle {s\atop s}\!>f-s-s-f$, $f-s-s-f-s$

$\displaystyle {{E_{11}'}\atop {E_{12}'}}\!> F_1- E_{13}'- E_{13}- G_3- E_{33}$

$\displaystyle {{E_{22}'}\atop {E_{23}'}}\!>F_2- E_{21}'-E_{21}- G_1$

$G_2- E_{32}-E_{32}'- F_3- E_{31}'$

(5) $D_7 \oplus D_9 \oplus A_2$: $\displaystyle {s\atop s}\!>f-s-s-f-s$, $\displaystyle {s\atop s}\!>f-s-s-f-s-s-f$, $f-s$

$\displaystyle {{E_{11}'}\atop {E_{12}'}}\!>F_1- E_{13}'-E_{13}- G_3- E_{33}$

$\displaystyle {{E_{22}'}\atop {E_{23}'}}\!>F_2- E_{21}'-E_{21}- G_1- E_{31}- E_{31}'- F_3$

$G_2-E_{32}$

(6) $D_{10}\oplus D_6 \oplus A_2$: $\displaystyle {s\atop
s}\!>f-s-s-f-s-s-f-s-s-f-s$, $\displaystyle {s\atop s}\!>f-s-s-f$,
$f-s$

$\displaystyle {{E_{11}'}\atop {E_{12}'}}\!>F_1- E_{13}'-E_{13}- G_3- E_{33}- E_{33}'- F_3- E_{31}'$

$\displaystyle {{E_{22}'}\atop {E_{23}'}}\!>F_2-E_{21}'-E_{21}- G_1$

$G_2-E_{32}$}

\begin{center}
  (B) Indeterminate Cases
\end{center}

{\footnotesize
\begin{tabbing}
  \=\hspace{8mm}\=Case V:\qquad \=(2)\quad \=$D_9\oplus D_9$: $\displaystyle {s\atop s}\!>f-s-s-f-s-s-f$, $\displaystyle
  {s\atop s}\!>f-s-s-f-s-s-f$\\[2mm]
  \>\>Case IX: \>(1) \>$A_3\oplus D_6\oplus D_9$: $s-f-s$, $\displaystyle {s\atop s}\!>f-s-s-f$, $\displaystyle
  {s\atop s}\!>f-s-s-f-s-s-f$\\[2mm]
  \>\>Case IX: \>(2) \>$A_6\oplus D_6\oplus D_6$: $s-f-s-s-f-s$, $\displaystyle {s\atop s}\!> f-s-s-f$,
  $\displaystyle {s\atop s}\!>f-s-s-f$\\[2mm]
  \>\>Case XII:\>(1) \>$D_4\oplus D_4\oplus A_{10}$: $\displaystyle {s\atop s}\!>f-s$, $\displaystyle {s\atop s}\!>f-s$,
  $f-s-s-f-s-s-f-s-s-f$\\[2mm]
  \>\>Case XII:\>(3) \>$D_4\oplus D_{10}\oplus A_4$: $\displaystyle {s\atop s}\!>f-s$, $\displaystyle {s\atop s}\!>f-s-s-f-s-s-f-s$,
  $f-s-s-f$\\[2mm]
  \>\>Case XII:\>(4) \>$D_4\oplus D_{13}\oplus A_1$: $\displaystyle {s\atop s}\!>f-s$, $\displaystyle {s\atop s}\!>
  f-s-s-f-s-s-f-s-s-f-s$, $f$\\[2mm]
  \>\>Case XIII:\>(1) \>$D_4\oplus D_6\oplus A_8$: $\displaystyle {s\atop s}\!>f-s$, $\displaystyle {s\atop s}\!>f-s-s-f$,
  $f-s-s-f-s-s-f-s$\\[2mm]
  \>\>Case XIII:\>(2) \>$D_4\oplus D_9\oplus A_5$: $\displaystyle {s\atop
  s}\!>f-s$, $\displaystyle {s\atop s}\!>f-s-s-f-s-s-f$,
  $f-s-s-f-s$.
\end{tabbing}
}

\begin{table}[h]
  \caption{$I=2,4$}\label{table2.2}
\end{table}

We use the same labeling as in Figure~\ref{fig2.2}. For $I=2$,
``$f$'' denotes the $g$-fixed curve and $s$
denotes the $g$-stable but not $g$-fixed curve. For $I=4$, define $h=g^2$; ``$f$'' denotes the $g$-fixed curve, ``$h$'' denotes the $h$-fixed but not
$g$-fixed curve and ``$s$'' denotes the $g$-stable but not $h$-fixed
curve.

{\footnotesize (1) $A_1\oplus A_{17}$:

$I=2$: $f$, $f-s-f-s-f-s-f-s-f-s-f-s-f-s-f-s-f$

$I=4$: $h$, $h-s-f-s-h-s-f-s-h-s-f-s-h-s-f-s-h$

$H_{11}$

$H_{13}-E_{13}'-F_1-E_{12}-G_2-E_{32}-F_3-E_{33}'-
H_{33}-G_3-E_{23}'-F_2-E_{21}'-G_1-E_{31}-H_{31}$.

(2) $A_3\oplus A_{15}$:

$I=2$: $f-s-f$, $f-s-f-s-f-s-f-s-f-s-f-s-f-s-f$

$F_2-E_{22}-G_2$

$H_{11}-E_{11}'-F_1-E_{13}'-H_{13}- E_{13}- G_3- E_{33}- H_{33}- E_{33}'- F_3- E_{31}'- H_{31}- E_{31}- G_1$.

(3) $A_5\oplus A_{13}$:

$I=2$: $f-s-f-s-f$, $f-s-f-s-f-s-f-s-f-s-f-s-f$

$I=4$: $h-s-f-s-h$, $h-s-f-s-h-s-f-s-h-s-f-s-h$

$H_{13}-E_{13}-G_3-E_{33}-H_{33}$

$H_{11}-E_{11}'-F_1-E_{12}-G_2-E_{32}- F_3-E_{31}'- H_{31}- E_{31}- G_1-E_2'-F_2$

(4) $A_7\oplus A_{11}$:

$I=2$: $f-s-f-s-f-s-f$, $f-s-f-s-f-s-f-s-f-s-f$

$H_{13}-E_{13}-G_3-E_{33}-H_{33}-E_{33}'-F_3$

$H_{11}-E_{11}'-F_1 - E_{12}- G_2- E_{22} - F_2- E_{21}'- G_1- E_{31}- H_{31}$

(5) $A_9\oplus A_9$:

$I=2$: $f-s-f-s-f-s-f-s-f$, $f-s-f-s-f-s-f-s-f$

$I=4$: $h-s-f-s-h-s-f-s-h$, $h-s-f-s-h-s-f-s-h$)

$H_{11}-E_{11}'-F_1 - E_{12}- G_2- E_{32}- F_3- E_{33}'- H_{33}$

$H_{13}- E_{13}- G_3-E_{23}'- F_2- E_{21}'- G_1- E_{31}- H_{31}$

}

\bibliographystyle{amsplain}
\bibliography{Log}

\end{document}